\documentclass[12pt, a4paper]{amsart}

\usepackage{microtype}

\usepackage[textwidth=391pt, marginparwidth=70pt, centering]{geometry}

\usepackage{graphicx}
\usepackage{xcolor}
\usepackage{tikz}
\usepackage{subcaption}

\usepackage{amssymb}
\usepackage{mathtools}
\usepackage{bm}
\usepackage{hhline}

\newtheorem{thm}{Theorem}[section]

\usepackage{aliascnt}	
\newaliascnt{lemma}{thm}
\newtheorem{lemma}[lemma]{Lemma}
\aliascntresetthe{lemma}

\newaliascnt{conj}{thm}
\newtheorem{conj}[conj]{Conjecture}
\aliascntresetthe{conj}

\newaliascnt{prop}{thm}
\newtheorem{prop}[prop]{Proposition}
\aliascntresetthe{prop}

\newaliascnt{corollary}{thm}
\newtheorem{corollary}[corollary]{Corollary}
\aliascntresetthe{corollary}

\newaliascnt{remark}{thm}
\theoremstyle{remark}
\newtheorem{remark}[remark]{Remark}
\aliascntresetthe{remark}

\theoremstyle{definition}
\newtheorem*{ack}{Acknowledgments}

\newcommand{\NN}{\mathbb{N}}
\newcommand{\ZZ}{\mathbb{Z}}
\newcommand{\QQ}{\mathbb{Q}}
\newcommand{\RR}{\mathbb{R}}
\newcommand{\CC}{\mathbb{C}}
\newcommand{\QQab}{\QQ^{\mathrm{ab}}}

\DeclareMathOperator{\Gal}{Gal}

\DeclareMathOperator{\Frob}{Frob}

\newcommand\house[1]{
  \begingroup
  \renewcommand{\arraystretch}{0.9}%
  \mathord{%
    \begin{array}{|@{\;}c@{\;}|}
      \hhline{|-|}
      #1
    \end{array}%
  }%
  \mkern1mu
  \endgroup
}

\usepackage[
	backend = biber,
	style   = alphabetic,
	backref = true,
	url     = false,
	doi     = true,
	isbn    = false,
]{biblatex}
\DefineBibliographyStrings{english}{
  backrefpage  = {$\uparrow$\hspace{-2.5pt}},
  backrefpages = {$\uparrow$\hspace{-2.5pt}},}

\DeclareFieldFormat[article]{title}{#1}

\DeclareFieldFormat[book, incollection, inproceedings, misc]
  {title}{\mkbibemph{#1}}

\DeclareFieldFormat{doi}{%
  \href{https://doi.org/#1}{\textsc{DOI}}%
}

\definecolor{unibas-red}{RGB}{210, 5, 55}
\usepackage[
  colorlinks,
  citecolor=unibas-red,
  linkcolor=unibas-red,
  urlcolor=.,
  linktocpage
]{hyperref} 

\usepackage[nameinlink]{cleveref}
\crefname{thm}{theorem}{theorems}
\Crefname{thm}{Theorem}{Theorems}

\crefname{conj}{conjecture}{conjectures}
\Crefname{conj}{Conjecture}{Conjectures}

\crefname{lemma}{lemma}{lemmas}
\Crefname{lemma}{Lemma}{Lemmas}

\crefname{prop}{proposition}{propositions}
\Crefname{prop}{Proposition}{Propositions}

\crefname{corollary}{corollary}{corollaries}
\Crefname{corollary}{Corollary}{Corollaries}

\crefname{remark}{remark}{remarks}
\Crefname{remark}{Remark}{Remarks}

\title[
  A relative Schinzel--Zassenhaus bound for $T^2 - 1$
  ]{
    A relative Schinzel--Zassenhaus bound for the polynomial $T^2 - 1$
    }
\date{\today}

\author[Alessio~Cangini]{Alessio~Cangini}
\address{
  University of Basel,
  Department of Mathematics and Computer Science,
  Spiegelgasse~$1$,
  CH-$4051$ Basel,
  Switzerland.}
\email{alessio.cangini@unibas.ch}
\urladdr{\url{https://sites.google.com/view/alessio-cangini/}}

\thanks{The author acknowledges support by the Swiss National Science 
Foundation Grant ``Rational points, arithmetic dynamics, and heights'' 
$\mathrm{n^\circ} 200020\_219397$.}

\begin{document}

\begin{abstract}
For $f(T) = T^2 - 1$, we establish a lower bound for the relative 
dynamical house of wandering algebraic integers. This bound depends 
on the degree over the maximal abelian extension of the rationals. 
Our proof follows Dimitrov's approach to the Schinzel--Zassenhaus 
Conjecture and its generalization by Habegger--Schmidt.
\end{abstract}

\maketitle

\section{Introduction}

Let $f \in \CC[T]$ be a polynomial of degree at least $2$. Set $f^0(T) 
= T$ and denote by $f^n$ the $n$th iterate of $f$ for all $n \in \NN 
= \{1, 2, \dots\}$. A number $z \in \CC$ is $f$\textit{-preperiodic} 
if the orbit $\{f^n(z) \colon n \ge 0\}$ is finite and $f$\textit{-wandering} 
otherwise. If $f$ is clear from the context, we will omit it from the 
notation. The \textit{canonical local height} of $z \in \CC$ is 
defined as
\[
    \hat \lambda_f(z)
    = \lim_{n \rightarrow \infty} \frac{\log^+ |f^n(z)|}{(\deg f)^n},
\]
where $|\cdot|$ denotes the standard complex absolute value and 
$\log^+ |z| = \log \max\{|z|, 1\}$. For all $z \in \CC$, we have 
$\hat\lambda_f(f(z)) = \deg f \cdot \hat\lambda_f(z)$; 
see~\cite[Proposition~5.58.(b)]{MR2316407}.

For $f(T) = T^2 - 1$, we establish the following explicit lower bound
for the logarithmic relative dynamical house of wandering algebraic 
integers depending on their degree over the maximal abelian extension 
$\QQab$ of $\QQ$.
\begin{thm}\label{thm:main-thm}
Let $f(T) = T^2 - 1$. Let $A \in \QQab[T]$ be monic and irreducible. 
Suppose that the coefficients of $A$ are algebraic integers and that 
its roots are wandering. Let $\zeta$ be a root of unity such that 
$A \in \ZZ[\zeta][T]$. 
Then
\begin{equation}\label{eqn:polynomial-main-thm-1}
    \frac{1}{[\QQ(\zeta):\QQ]}
    \sum_{\tau \colon \QQ(\zeta)\rightarrow\CC}
    \max_{\beta \colon \tau A(\beta)=0}
    \hat\lambda_f(\beta)
    \ge
    \frac{\log(2^{1/5}(\sqrt{5} -1))}{240 \deg A}.
\end{equation}
\end{thm}
The polynomial $\tau A$ is defined by applying $\tau$ to the 
coefficients of~$A$. 

If $A$ is the minimal polynomial of an algebraic integer $\alpha$ 
over $\QQab$, we call the left-hand side of
\eqref{eqn:polynomial-main-thm-1} the \textit{logarithmic relative 
dynamical house} of~$\alpha$. Since the logarithmic relative 
dynamical house is invariant along the Galois orbit, we will refer to 
it also as the logarithmic relative dynamical house of $A$. We will 
prove in~\Cref{lemma:2.1} that it does not depend on the choice of a 
field containing the coefficients of~$A$. 

In the literature, $\QQab$ provides a fruitful setting to test 
Lehmer-type problems; see, for example,~\cite{MR1740514, MR1817715, 
MR2067482, MR2029512, MR2098701, MR2651944}.

The root of unity in the statement of~\Cref{thm:main-thm} exists by 
the Kronecker--Weber theorem~\cite[Theorem~V.1.10]{MR1697859}. For 
this reason, in what follows we will often reduce to the cyclotomic 
case.

\subsection{Context}

The study of algebraic integers close to the unit disk is a classical 
topic in number theory. In these notes, we replace the unit disk with 
the \textit{filled Julia set} of $f(T) = T^2 - 1$. 

The \textit{filled Julia set} of a polynomial $f \in \CC[T]$ of 
degree at least~$2$ is the set~$\mathrm{K}_f$ of complex numbers with 
bounded orbit. For example, the filled Julia set of $f(T) =T^2$ is 
the unit disk and the filled Julia set of $f(T) = T^2 - 2$ is the 
segment $[-2, 2]$. The filled Julia set of $f(T) = T^2 - 1$ is 
displayed in \Cref{fig:figure1}. In general, the filled Julia set 
corresponds with the vanishing locus of $\hat\lambda_f$; 
see~\cite[Ch. 18]{MR2193309}. 
\begin{figure}[h]
    \centering
    \includegraphics[width=7cm]{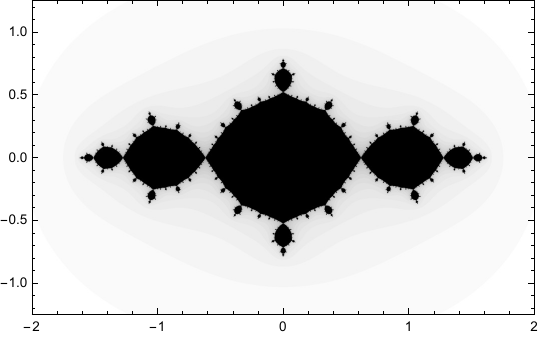}
    \caption{\small The filled Julia set of $f(T) = T^2 - 1$ in 
    black. The sublevel sets of $\hat \lambda_f$ in gray. The picture 
    was produced with \textsc{Wolfram Mathematica}.} 
    \label{fig:figure1}
\end{figure}

If the coefficients of $f$ are algebraic integers in a number field 
$K$ and $\alpha \in \overline\QQ$ is an algebraic integer, the 
logarithmic dynamical house of $\alpha$
\[
    \max_{\sigma \colon K(\alpha) \rightarrow \CC}\hat \lambda_{\sigma(f)}(\sigma(\alpha))
\] 
is positive if and only if $\alpha$ is wandering; cf.~\cite[Theorems~3.5 and 3.22]{MR2316407}.
This quantity was investigated by Habegger--Schmidt in~\cite{MR4726503}. 
By analogy with the Schinzel--Zassenhaus Conjecture~\cite{MR0175882}, 
we expect it to decay inversely proportional to the degree. 
Therefore, we propose the following dynamical analog.

\begin{conj}[Dynamical Schinzel--Zassenhaus]\label{conj:dyn-sz}
Let $K$ be a number field with ring of integers $O_K$. Let $f \in O_K[T]$ 
have degree at least $2$. There exists a constant $C = C(K, f) > 0$ 
such that
\[
    \max_{\sigma \colon K(\alpha) \rightarrow \CC}
    \hat \lambda_{\sigma(f)}(\sigma(\alpha))
    \ge
    \frac{C}{[K(\alpha) \colon K]}
\]
for all $f$-wandering algebraic integers $\alpha \in \overline \QQ$.
\end{conj}

The case $f(T) = T^2$ corresponds to the Schinzel--Zassenhaus 
Conjecture, and was solved by Dimitrov~\cite{dimitrov}. Other 
contributions include work by Boyd~\cite{MR790657}, Dubickas~\cite{MR1481371} 
and Pritsker~\cite{MR4395356}. Following Dimitrov's ideas, Habegger 
and Schmidt~\cite{MR4726503} proved~\Cref{conj:dyn-sz} for certain 
postcritically finite unicritical polynomials of prime degree defined 
over suitable number fields. Recently, they extended 
in~\cite{habegger2026dynamicalcanonicalheightsfinite} their results 
to a more general class of polynomials.

\subsection{Motivation}

One important motivation to study the dynamical house is to derive 
bounds for the associated Call--Silverman height. 

Let $K$ be a number field with ring of integers $O_K$. Let $f 
\in O_K[T]$ be monic and with degree at least $2$. In this setting, 
the \textit{Call--Silverman height} of an algebraic integer $\alpha 
\in \overline\QQ$ is
\[
    \hat h_f(\alpha)
    = \frac{1}{[K(\alpha) \colon \QQ]}
    \sum_{\sigma \colon K(\alpha) \rightarrow \CC}
    \hat \lambda_{\sigma(f)}(\sigma(\alpha)).
\]
The dynamical Kronecker theorem implies that an algebraic integer is 
$f$-preperiodic if and only if it is a zero of $\hat h_f$; see 
also~\cite[Theorem~3.22]{MR2316407}. The Dynamical Lehmer Conjecture 
predicts the best decay for the Call--Silverman height. In this 
setting, it can be stated as follows.

\begin{conj}[Dynamical Lehmer Conjecture]\label{conj:dyn-l}
Let $K$ be a number field with ring of integers $O_K$. Let $f 
\in O_K[T]$ be monic and with degree at least $2$. There exists a 
constant $C = C(K, f) > 0$ such that $\hat h_f(\alpha) \ge C / [K(\alpha) \colon K]$ 
for all wandering algebraic integers $\alpha \in \overline \QQ$.
\end{conj}
For a more general statement, see~\cite[Conjecture~3.25]{MR2316407}.
\Cref{conj:dyn-l} implies \Cref{conj:dyn-sz}. As a corollary 
of~\Cref{thm:main-thm}, we obtain the following lower bound for the 
Call--Silverman height of algebraic integers.

\begin{corollary}\label{cor:height-bound}
Let $f(T) = T^2 - 1$. Then for all wandering algebraic integers 
$\alpha \in \overline\QQ$ one has that
\begin{equation}\label{eqn:height-bound1}
    \hat h_f(\alpha)
    \ge \frac{\log(2^{1/5}(\sqrt{5} -1))}{240[\QQab(\alpha) \colon \QQab]^2}.
\end{equation}
\end{corollary}

\subsection{Method}

Our proof uses tools from Dimitrov's proof strategy for the 
Schinzel--Zassenhaus Conjecture~\cite{dimitrov}, its adaptation to the
dynamics due to Habegger--Schmidt~\cite{MR4726503} and the extension 
to the case of cyclotomic extensions that the author studied 
in~\cite{cangini2026lowerboundcallsilvermanheight}. 

Let $f(T) = T^2 - 1$. Let $A \in \QQab[T]$ be monic and irreducible, 
and suppose that its coefficients are algebraic integers. Let $\zeta$ 
be a root of unity such that $A \in \ZZ[\zeta][T]$. Factor $A$ as
\[
    A(T) = (T - \alpha_1) \cdots (T - \alpha_D)
\]
and for all integers $n \ge 0$ define 
\[
    A_n(T) = (T - f^n(\alpha_1)) \cdots (T - f^n(\alpha_D)).
\]
By~\Cref{lemma:coeff-An}, $A_n \in \ZZ[\zeta][T]$.

The basic idea is to consider a square root $\Phi$ of $A_m / \sigma 
A_1$ where $\sigma \in \Gal(\QQ(\zeta) / \QQ)$ and $m \in \NN$ are 
chosen so that $\Phi \in \ZZ[\zeta][\![1 / T]\!]$. Write $\widehat\CC 
= \CC \cup \{\infty\}$. Habegger--Schmidt constructed a simply 
connected domain $U \subseteq \widehat \CC$ such that $\widehat \CC 
\setminus U$ contains zeros and poles of $A_m / \sigma A_1$. They 
also provided an upper bound for the transfinite diameter of this 
complement. If $\Phi$ is not a rational function,~\Cref{thm:polya-bertrandias} 
yields a lower bound for the transfinite diameter, and the proof 
follows from a comparison of the two bounds.

As mentioned, the argument works if $\Phi$ is not rational; in 
particular the rational case has to be treated separately. The 
abelian setting implies two different analyses according to the 
ramification of $2$ in $\QQ(\zeta)$. While the unramified case does 
not present particular problems, the ramified case is more~delicate. 

In~\cite{cangini2026lowerboundcallsilvermanheight}, the main argument 
to rule out the possibility that $\Phi$ is rational in the ramified 
case was based on the non-degeneracy of the conductor after 
iteration. Concretely, reducing to the current setting in which $f(T) 
= T^2 - 1$, we proved that if $2$ is ramified in $\QQ(\zeta^4)$, then 
for all algebraic integers $\alpha \in \ZZ[\zeta] \setminus 
\ZZ[\zeta^2]$ such that $f(\alpha) \in \ZZ[\zeta^2]$ it holds that 
$f^2(\alpha) \notin \ZZ[\zeta^4]$.

In the present relative setting, the analogous claim would be the 
following: if~$2$ is ramified in $\QQ(\zeta^4)$, then for all monic 
irreducible $A \in \ZZ[\zeta][T] \setminus \ZZ[\zeta^2][T]$ such that 
$A_1 \in \ZZ[\zeta^2][T]$, it holds that $A_2 \notin \ZZ[\zeta^4][T]$. 
In~\Cref{lemma:ramified-lemma4}, we prove that this claim is true 
if~$D$ is odd. However, it turns out to be false in general as the 
following counterexample shows.

\begin{prop}\label{prop:example}
Let $\zeta$ be a primitive $16$th root of unity. There exists a monic 
irreducible $A \in \ZZ[\zeta][T] \setminus \ZZ[\zeta^2][T]$ with 
wandering roots such that $A_1 \in \ZZ[\zeta^2][T]$ and 
$A_2 \in \ZZ[\zeta^4][T]$.
\end{prop}
The polynomial in the following counterexample has been generated 
with the help of \textsc{GPT}-5.6 Sol. 

\begin{proof}
Fix an embedding of $\overline\QQ$ into $\CC$. Write $i = \zeta^4$. 
Note that $2$ is ramified in $\QQ(i)$. Consider
\[
    A(T) = T^4 + (\zeta + \zeta^5)T^3 - 2 T^2 -(\zeta + \zeta^5)T + (1 + \zeta^6).
\]
The automorphism of $\QQ(\zeta)$ mapping $\zeta$ to $-\zeta$ maps 
$\zeta + \zeta^5$ to its opposite. Hence $A \notin \ZZ[\zeta^2][T]$.

A direct computation gives 
\[
    A_1(T) =
    T^4
    - 2 \zeta^6 T^3
    - \zeta^4 \in \ZZ[\zeta^2][T]
\]
and
\[
    A_2(T) =
    T^4
    + 4(i + 1) T^3
    + 2(5i + 3) T^2
    + 4(2i + 1) T
    + 2i \in \ZZ[\zeta^4][T].
\]

We next prove that $A$ is irreducible over $\QQab$. The reduction of  
$A$ 
modulo $\mathfrak p = (17, \zeta - 3)$ is $T^4 + 8T^3 - 2 T^2 - 8T 
-1$, which is irreducible over $\ZZ[\zeta]/\mathfrak p \cong 
\mathbb{F}_{17}$. Thus, $A$ is irreducible over $\QQ(\zeta)$. It 
follows that the Galois group $G$ of $A$ over $\QQ(\zeta)$ is a 
transitive subgroup of~$S_4$. The cubic resolvent of $A$ is
\[
    R(T) 
    = T^3 + 2 T^2 - (6\zeta^6 + 4) T - 12 \zeta^6 + 2 \zeta^4 - 8.
\]
The reduction of $R$ modulo $\mathfrak{q} = (97, \zeta - 12)$ is 
$T^3 + 2T^2 - 8T + 37$, which is irreducible in $\ZZ[\zeta]/\mathfrak{q} 
\cong \mathbb{F}_{97}$. Hence $R$ is irreducible over $\QQ(\zeta)$. 
It follows from~\cite[Theorem 13.1.1.(a)]{MR2919975} that $G$ is 
either $A_4$ or~$S_4$. Therefore, the commutator
\[
    [G, G] =
    \begin{cases}
    V_4 & \text{ if } G = A_4,\\
    A_4 & \text{ if } G = S_4
    \end{cases}
\]
is a transitive subgroup of $S_4$.

Denote by $K$ the splitting field of $A$ over $\QQ(\zeta)$. Since 
$K\QQab$ is the splitting field of $A$ over $\QQab$, the polynomial 
$A$ is irreducible over $\QQab$ if and only if~$\Gal(K\QQab / \QQab)$ 
is a transitive subgroup of $S_4$.
We prove that the group $\Gal(K / K \cap \QQab)$ is a transitive 
subgroup of $S_4$. Note that
\[
    \Gal(K \cap \QQab / \QQ(\zeta))
    \cong \frac{G}{\Gal(K/ K \cap \QQab)}.
\]
Given that the left-hand side is abelian, it follows that $[G, G] 
\subseteq \Gal(K/ K \cap \QQab)$. Since~$[G, G]$ is a transitive 
subgroup of $S_4$, the same follows for $\Gal(K/ K \cap \QQab)$ as 
well.
Restriction to $K$ induces the isomorphism 
\[
    \Gal(K\QQab / \QQab) \cong \Gal(K / K \cap \QQab). 
\]    
Therefore, $\Gal(K\QQab / \QQab)$ is a 
transitive subgroup of $S_4$. We conclude that~$A$ is irreducible 
over $\QQab$.

It remains to show that the roots of $A$ are wandering. Since $A$ is 
irreducible and $f$ has rational coefficients, the transitivity of 
the Galois action implies that if one root of $A$ is wandering, then 
all its roots are wandering. As every root of $A_1$ is the image 
under $f$ of a root of $A$, it is enough to find a wandering root of 
$A_1$. 

In particular, it is enough to find a root of $A_1$ with absolute 
value greater than $2$. Indeed, the reverse triangle inequality 
implies that 
\[
    |f(\beta)| 
    \ge |\beta|^2 - 1 > |\beta| > 2, 
\]
from which we conclude that $\beta$ is wandering since its absolute 
value is increasing along the orbit. 

Since $|\zeta^6| = 1$, finding a root of $A_1$ with absolute value 
greater than~$2$ is equivalent to finding such a root of $B(T) =
- A_1(\zeta^6 T) = T^4 - 2 T^3 + i$. 

Denote the roots of $B$ by $\beta_1, \dots, \beta_4$. Note that 
$B(2) = i$ and that $B'(2) = 8$. The logarithmic derivative gives
\[
   \sum_{k = 1}^4 \frac{1}{2 - \beta_k} 
   = \frac{B'(2)}{B(2)}
   = -8i.
\]
Taking the real part yields
\[
    \sum_{k = 1}^4 
    \frac{2 - \mathrm{Re}(\beta_k)}{|2 - \beta_k|^2} = 0.
\]
Suppose, for contradiction, that $|\beta_k| \le 2$ for all 
$k \in \{1, \dots, 4\}$. Then $\mathrm{Re}(\beta_k) \le~2$, 
so every summand in the last sum is nonnegative. Hence every summand 
must vanish, so $\mathrm{Re}(\beta_k) = 2$ for all $k \in 
\{1, \dots, 4\}$. Together with $|\beta_k| \le 2$, this 
forces $\beta_k = 2$, contradicting $B(2) = i \neq 0$. Thus $B$, and
hence $A_1$, has a root $\beta$ of absolute value greater than $2$. 
\end{proof}

For even $D$,~\Cref{lemma:ramified-case-hard} provides a Galois
automorphism for which the strategy applies, leading eventually to 
the proof.

\subsection*{Outline}

In~\Cref{sec:properties-rel-dyn-house}, we establish some relevant 
properties of the auxiliary polynomials and of the relative dynamical 
house of a polynomial. The function $\Phi$ is defined 
in~\Cref{sec:polynomial-congruences} after proving the polynomial 
congruences needed to guarantee the integrality of the coefficients. 
In~\Cref{sec:dimitrov-background} we recall the Habegger--Schmidt and 
P\'olya--Bertrandias theorems, together with the necessary background 
about the transfinite diameter. Finally, the proof 
of~\Cref{thm:main-thm} is given in~\Cref{sec:proof}.

\subsection*{AI usage}

Every line of this paper was written by a human. GPT-5.6 Sol was used 
to produce the example in~\Cref{prop:example} and to proofread 
the~paper.


\begin{ack}
I am very grateful to my PhD advisor Philipp Habegger for suggesting 
this topic, for the many helpful discussions on it, and for the 
precious comments on the first draft on these paper.
\end{ack}
\section{Properties of the auxiliary polynomials and of the relative dynamical house}\label{sec:properties-rel-dyn-house}

Let $f(T) = T^2 - 1$. Let $\zeta$ be a root of unity.

\subsection{The auxiliary polynomial}
 
Let $K$ be a number field with ring of integers $O_K$. Let $A \in K[T]$ 
be a monic polynomial of degree $D \ge 1$. Denote its roots by 
$\alpha_1, \dots, \alpha_D$. For all~$n \ge 0$, define
\[
	A_n(T) 
    = \left( T - f^n(\alpha_1) \right) 
    \cdots 
    \left( T - f^n(\alpha_D) \right).
\]

We first establish some elementary properties of $A_n$.

\begin{lemma}\label{lemma:coeff-An}
	Let $f(T) = T^2 - 1$. If $A \in O_K[T]$ is a nonconstant monic 
    polynomial, then $A_n \in O_K[T]$ for all $n \in \NN$.
\end{lemma}
\begin{proof}
Denote the roots of $A$ by $\alpha_1, \dots, \alpha_D$. Since $
f^n \in \ZZ[T]$, each coefficient of $A_n$ is a symmetric polynomial 
in $f(\alpha_1), \dots, f(\alpha_D)$ with integer coefficients. By 
the fundamental theorem of symmetric polynomials, it is therefore an 
integer polynomial in the elementary symmetric functions in the roots 
of~$A$. Since these are, up to a sign, the coefficients of $A$, we 
conclude that $A_n \in~O_K[T]$.
\end{proof}

\begin{lemma}\label{lemma:prop-An}
	Let $f(T) = T^2 - 1$. Let $A, B \in K[T]$ be nonconstant monic 
    polynomials. The following statements hold true.
	\begin{enumerate}
		\item[(i.)] We have $(A B)_n = A_n B_n$ for all $n \in 
        \NN \cup \{0\}$.
		\item[(ii.)] We have $(A_n)_m = A_{n + m}$ for all $n, m \in 
        \NN \cup \{0\}$. 
	\end{enumerate}
\end{lemma}
\begin{proof}
Let $\alpha_1, \dots, \alpha_D$ be the roots of $A$ and $\beta_1, 
\dots, \beta_E$ be the roots of $B$, respectively listed with 
multiplicity. By definition
\[
	A_n(T) = \prod_{i = 1}^D (T - f^n(\alpha_i))
\]
and 
\[
	B_n(T) = \prod_{i = 1}^E (T - f^n(\beta_i)).
\]
The roots of $AB$ listed with multiplicity are $\alpha_1, \dots, 
\alpha_D, \beta_1, \dots, \beta_E$. Hence
\[
	(AB)_n(T) 
    = \prod_{i = 1}^D (T - f^n(\alpha_i))
    \prod_{i = 1}^E (T - f^n(\beta_i))
	= A_n(T) B_n(T),
\]
which proves (i.).

Similarly, the roots of $A_n$ listed with multiplicity are 
$f^n(\alpha_1), \dots, f^n(\alpha_D)$. Therefore
\[
	(A_n)_m(T)
	= \prod_{i = 1}^D (T - f^m(f^n(\alpha_i)))
	= \prod_{i = 1}^D (T - f^{m + n}(\alpha_i))
    = A_{n + m}(T),
\]	
which proves (ii.).
\end{proof}

\begin{lemma}\label{lemma:vanishing-order-A_n}
Let~$K$ be a number field and let~$A\in K [T]$ be monic and 
irreducible. For every $n \ge 1$, the roots of~$A_n$ have the same 
multiplicity.
\end{lemma}
\begin{proof}
Fix $n \ge 1$. Let~$S_0$ be the zero set of~$A$ and~$S_n$ be the zero 
set of~$A_n$. If~$|S_n| = 1$, we are done, so we may assume~$|S_n| > 1$. 
Consider the surjective map~$f^n|_{S_0} \colon S_0 \rightarrow S_n$. 
Since $A$ is separable, the multiplicity of $\beta \in S_n$ equals 
$|{(f^n|_{S_0})}^{-1}(\beta)|$.

The Galois group~$G_K = \Gal(\overline \QQ / K)$ acts transitively 
on~$S_0$ since~$A$ is irreducible. The action of~$G_K$ on~$S_n$ is 
also transitive. Indeed, given $\beta_1, \beta_2 \in S_n$, choose 
$\alpha_1, \alpha_2 \in S_0$ such that $f^n(\alpha_i) = \beta_i$ for 
$i \in \{1, 2\}$. By transitivity of the action on $S_0$, there 
exists $\tau \in G_K$ such that $\tau(\alpha_1) = \alpha_2$. Since 
$f \in \QQ[T]$, we have 
\[
	\tau(\beta_1) 
	= \tau(f^n(\alpha_1)) 
	= f^n(\tau(\alpha_1))
	= f^n(\alpha_2)
	= \beta_2.
\] 

Consider any two distinct~$\beta_1, \beta_2 \in S_n$. By 
transitivity, there exists~$\tau \in G_K$ such that~$\tau(\beta_1) = 
\beta_2$. Hence,~$\tau$ defines a bijective map from~${(f^n|_{S_0})}^{-1}(\beta_1)$ 
to~${(f^n|_{S_0})}^{-1}(\beta_2)$ whose inverse is induced by~$\tau^{-1}$. 
This concludes the proof.
\end{proof}

The following lemma shows that if $A$ is irreducible over $\QQab$, 
then $A_1$ is either irreducible or the square of an irreducible 
polynomial.

\begin{lemma}\label{lemma:min-poly}
Let $\alpha \in \overline{\QQ}$ and denote by $A$ its minimal 
polynomial over $\QQab$. Then either $A_1$ is the minimal polynomial 
of $f(\alpha)$ over $\QQab$ or it is its square.
\end{lemma}
\begin{proof}
Denote the degree of $A$ by $D$. The polynomial $T^2 - \alpha^2 \in 
\QQab(f(\alpha))[T]$ vanishes at $\alpha$, so $[\QQab(\alpha) \colon 
\QQab(f(\alpha))] \in \{1, 2\}$. By the multiplicativity of the 
degree of an extension, it follows that $[\QQab(f(\alpha)) \colon 
\QQab] \in \{D, D / 2\}$. Since $f(\alpha)$ is a root of $A_1 \in 
\QQab[T]$, the minimal polynomial of~$f(\alpha)$ over $\QQab$ divides 
$A_1$. If $A_1$ is irreducible over $\QQab$, the first conclusion 
follows. If $A_1$ is reducible, then we must have 
$[\QQab(f(\alpha)) \colon \QQab] = D/2$. Note that $G_{\QQab} = 
\Gal(\overline\QQ / \QQab)$ acts transitively on the roots of $A_1$, 
as it does on the roots of $A$. This implies that $A_1 = B^k$ for 
some $k \ge 1$, where $B \in \QQab[T]$ is the minimal polynomial of 
$f(\alpha)$ over $\QQab$. Taking degrees gives
\[
    D = \deg A_1 = k \deg B = k D / 2,
\]
so $k = 2$. The second conclusion follows.
\end{proof}

\subsection{The relative dynamical house}
Let $f(T) = T^2 - 1$ and let $\zeta$ be a root of unity. If $A \in 
\ZZ[\zeta][T]$ is monic and nonconstant, its \textit{relative 
dynamical house} is defined as
\[
    \house{A_{}}_f^\mathrm{ab}
    = \exp \left(
        \frac{1}{[\QQ(\zeta) \colon \QQ]}
        \sum_{\sigma \colon \QQ(\zeta) \rightarrow \CC}
        \max_{\beta \colon \sigma A(\beta) = 0} 
        \hat\lambda_{f}(\beta)
        \right),
\]
where $\sigma \colon \QQ(\zeta) \rightarrow \CC$ is a field embedding 
and $\sigma A$ is defined by applying $\sigma$~coefficientwise.
First, we prove that this quantity is well defined.

\begin{lemma}\label{lemma:2.1}
Let $K$ be a number field and let $A \in K[T]$ be nonconstant. Then, 
for every finite extension $L/K$, it holds that
\[
    \frac{1}{[K \colon \QQ]}
    \sum_{\sigma \colon K \rightarrow \CC}
    \max_{\beta \colon \sigma A(\beta) = 0} \hat\lambda_{f}(\beta)
    =\frac{1}{[L \colon \QQ]}
    \sum_{\sigma \colon L \rightarrow \CC}
    \max_{\beta \colon \sigma A(\beta) = 0} \hat\lambda_{f}(\beta).
\]
In particular, the relative dynamical house is independent of the 
choice of a finite number field over which $A$ is defined.
\end{lemma}
\begin{proof}
Every embedding $\sigma \colon K \rightarrow \CC$ has exactly 
$[L \colon K]$ extensions to $L$. If~$\tau \colon L \rightarrow \CC$ 
is an embedding extending $\sigma \colon K \rightarrow \CC$, it 
follows that $\tau A = \sigma A$. Hence
\begin{align*}
    \frac{1}{[K \colon \QQ]}
    \sum_{\sigma \colon K \rightarrow \CC}
    \max_{\beta \colon \sigma A(\beta) = 0} \hat\lambda_{f}(\beta)
    &= \frac{1}{[K \colon \QQ]}
    \sum_{\sigma \colon K \rightarrow \CC}
    \frac{[L \colon K]}{[L \colon K]}
    \max_{\beta \colon \sigma A(\beta) = 0} \hat\lambda_{f}(\beta)\\
    &= \frac{1}{[L \colon \QQ]}
    \sum_{\sigma \colon K \rightarrow \CC}
    \sum_{\substack{\tau \colon L \rightarrow \CC \\ \tau|_{K} = \sigma}}
    \max_{\beta \colon \tau A(\beta) = 0} \hat\lambda_{f}(\beta)\\
    &= \frac{1}{[L \colon \QQ]}
    \sum_{\sigma \colon L \rightarrow \CC}
    \max_{\beta \colon \sigma A(\beta) = 0} \hat\lambda_{f}(\beta).\qedhere
\end{align*}
\end{proof}

Second, we relate the relative house of a polynomial to those of 
its associated auxiliary polynomials.
\begin{lemma}\label{lemma:functoriality-house}
For all $n \ge 0$, we have $\log \house{A_n}_f^\mathrm{ab} = 
2^n \log \house{A_{}}_f^\mathrm{ab}$.
\end{lemma}
\begin{proof}
If $n = 0$, there is nothing to prove, so assume $n \ge 1$. For all 
$n \ge 1$, the roots of $A_n$ are $\{f^n(\alpha) \colon A(\alpha) = 0\}$. 
Recall that $\hat\lambda_f(f^n(\alpha)) = 2^n \hat\lambda_f(\alpha)$. 
We have
\begin{align*}
    \log \house{A_n}_f^\mathrm{ab}
    &= \frac{1}{[\QQ(\zeta) \colon \QQ]}
    \sum_{\sigma \colon \QQ(\zeta) \rightarrow \CC}
    \max_{\beta \colon \sigma A_n(\beta) = 0} \hat\lambda_{f}(\beta)\\
    &= \frac{1}{[\QQ(\zeta) \colon \QQ]}
    \sum_{\sigma \colon \QQ(\zeta) \rightarrow \CC}
    \max_{\beta \colon \sigma A(\beta) = 0} \hat\lambda_{f}(f^n(\beta))\\
    &= \frac{2^n}{[\QQ(\zeta) \colon \QQ]}
    \sum_{\sigma \colon \QQ(\zeta) \rightarrow \CC}
    \max_{\beta \colon \sigma A(\beta) = 0} \hat\lambda_{f}(\beta)\\
    &= 2^n \log \house{A_{}}_f^\mathrm{ab}.\qedhere
\end{align*}
\end{proof}
\section{Congruence conditions and 
 the power series}\label{sec:polynomial-congruences}

Let $\zeta$ be a root of unity. The following proposition summarizes 
some properties of cyclotomic fields.
\begin{prop}\label{prop:cyclotomic-fields} \hfill
	\begin{enumerate}
		\item[(i.)] The cyclotomic field $\QQ(\zeta)$ is ramified 
		above $2$ if and only if the order of $\zeta$ is a multiple 
		of $4$.
		\item[(ii.)] The ring of integers of $\QQ(\zeta)$ is 
		$\ZZ[\zeta]$.
		\item[(iii.)] If $2$ is ramified in $\QQ(\zeta)$, then 
		$\QQ(\zeta) / \QQ(\zeta^2)$ is a quadratic extension.
		\item[(iv.)] If $2$ is ramified in $\QQ(\zeta^4)$, then 
		$\QQ(\zeta) / \QQ(\zeta^4)$ is a quartic cyclic extension.
	\end{enumerate}
\end{prop}
\begin{proof}
For (i.) and (ii.), see~\cite[Corollary I.10.4]{MR1697859} 
and~\cite[Proposition I.10.2]{MR1697859} respectively.

Suppose that $2$ is ramified in $\QQ(\zeta)$. By (i.), the order $N$ 
of $\zeta$ is a multiple of $4$. So
\[
	[\QQ(\zeta) \colon \QQ(\zeta^2)] 
	= \frac{\phi_E(N)}{\phi_E(N/2)} 
	= 2,
\]
where $\phi_E$ is Euler's totient function. This proves (iii.).

Suppose that $2$ is ramified in $\QQ(\zeta^4)$. Then, the order of 
$\zeta^4$ is divisible by~$4$ by (i.). Thus the order of $\zeta$ is divisible by $16$. It follows that 
\[
	[\QQ(\zeta) \colon \QQ(\zeta^4)]
	= \frac{\phi_E(N)}{\phi_E(N/4)}
	= 4.
\]
Moreover, $\QQ(\zeta^4)$ contains a primitive fourth root of unity. 
Since $\QQ(\zeta)$ is obtained by adding a fourth root of $\zeta^4$ 
to $\QQ(\zeta^4)$, Kummer's theory implies that $\QQ(\zeta) / 
\QQ(\zeta^4)$ is cyclic; see~\cite[Ch. IV.3]{MR1697859}. This 
proves~(iv.).
\end{proof}

\subsection{Congruences}

Let $f(T) = T^2 - 1$. Let $A \in \ZZ[\zeta][T]$ be monic of 
degree~$D \ge 1$. Recall that we defined
\[
	A_n(T) 
	= \left( T - f^n(\alpha_1) \right) 
	\cdots 
	\left( T - f^n(\alpha_D) \right),
\]
where $\alpha_1, \dots, \alpha_D$ denote the roots of $A$. The goal 
of this subsection is to establish congruence conditions 
modulo~$4\ZZ[\zeta][T]$ for the polynomials~$A_n$. 

If $2$ is unramified in $\QQ(\zeta)$, let $\Frob \in 
\Gal(\QQ(\zeta) / \QQ)$ denote the Frobenius automorphism at $2$, 
which is the unique automorphism such that
\begin{equation}\label{eqn:frob}
	\Frob(\alpha) \equiv \alpha^2 \mod 2 \ZZ[\zeta]
\end{equation}
for all~$\alpha \in \ZZ[\zeta]$. 

If instead $2$ is ramified in $\QQ(\zeta)$, then the extension 
$\QQ(\zeta) / \QQ(\zeta^2)$ is quadratic 
by~\Cref{prop:cyclotomic-fields}.(iii.). Let~$\kappa$ be the 
generator of~$\Gal(\QQ(\zeta)/\QQ(\zeta^2))$. We first claim that 
$\kappa(\zeta) = - \zeta$. Indeed, since $\kappa(\zeta^2) = \zeta^2$, 
we have $\kappa(\zeta) = \pm \zeta$. As~$\kappa$ does not 
fix~$\zeta$, it follows that $\kappa(\zeta) = - \zeta$. 

Then we claim that 
\begin{equation}\label{eqn:kappa}
	\kappa(\alpha) 
	\equiv 
	\alpha 
	\mod 2 \ZZ[\zeta] 
\end{equation}
for all $\alpha \in \ZZ[\zeta]$. Indeed, write~$\alpha = a_0 + 
a_1 \zeta + \cdots + a_{n} \zeta^{n}$ for~$a_0, \dots, a_{n} \in 
\ZZ$. Then
\[
	\kappa(\alpha) - \alpha 
	= - 2\left(
		a_1 \zeta 
		+ a_3 \zeta^3 
		+ \cdots + a_{n'} \zeta^{n'} \right) \in 2 \ZZ[\zeta]
\]
where $n' = n$ if $n$ is odd and $n' = n - 1$ otherwise.

The following proposition is the main result of this section.
\begin{prop}\label{prop:congruences}
Let~$(\sigma, m) \in \Gal(\QQ(\zeta) / \QQ) \times \NN$ be~$(\Frob^2, 3)$ 
if~$2$ is unramified in $\QQ(\zeta)$ and~$(\kappa, 1)$ otherwise. 
If $A \in \ZZ[\zeta][T]$ is monic, then
\[
	\sigma A_1 \equiv A_m \mod 4 \ZZ[\zeta][T].
\]
\end{prop}

\begin{remark}
Congruences~\eqref{eqn:frob} and~\eqref{eqn:kappa} can be written 
uniformly with the notation of~\Cref{prop:congruences} as
\begin{equation}\label{equation:numerical-congruence}
	\sigma(\alpha) \equiv \alpha^{2^{m-1}} \mod 2\ZZ[\zeta]
\end{equation}
for all~$\alpha \in \ZZ[\zeta]$.
\end{remark}

Denote by~$e_i$ the~$i$th elementary symmetric polynomial in~$D$ 
variables. Recall that~$e_i(f^n(\alpha_1), \dots, f^n(\alpha_D)) \in 
\ZZ[\zeta]$ for all $n \ge 0$. To prove~\Cref{prop:congruences} we 
need to show that 
\[
	\sigma \left(e_i \left( f(\alpha_1), \dots, f(\alpha_D) \right)\right)
	\equiv e_i \left( f^m(\alpha_1), \dots, f^m(\alpha_D) \right) 
	\mod 4\ZZ[\zeta]
\]
for all~$i \in \{ 1, \dots, D \}$.

Define the polynomial
\[
	\Psi_i = \frac{
		e_i(T_1, \dots, T_D)^2 
		- e_i(T_1^2, \dots, T_D^2)
	}{2}
\]
for $i \in \{1, \dots, D\}$. Since $e_i(T_1, \dots, T_D)^2 \equiv 
e_i(T_1^2, \dots, T_D^2) \mod 2\ZZ[T_1, \dots, T_D]$, it follows that 
$\Psi_i \in \ZZ[T_1, \dots, T_D]$.

\begin{lemma}{\emph{\cite[Lemma 2.1]{MR4726503}}}\label{lemma:poly-cong-lemma1}
For all~$n, D \in \NN$ and~$i\in\{1, \dots, D\}$ we have
\[
	{e_{i}(T_1, \dots, T_D)}^{2^n}
	\equiv {e}_{i}(T_1^{2^n},\dots,T_D^{2^n}) + 2\Psi_{i}^{2^{n-1}}
	\mod 4\ZZ[T_1, \dots, T_D].
\]
\end{lemma}

\begin{lemma}\label{lemma:poly-cong-lemma2}
Let~$(\sigma, m)$ be as in~\Cref{prop:congruences}. Then 
\[
	{\sigma(\alpha)}^2 \equiv \alpha^{2^m} \mod 4\ZZ[\zeta]
\]
for all~$\alpha \in \ZZ[\zeta]$.
\end{lemma}
\begin{proof}
By~\eqref{equation:numerical-congruence} there exists~$\beta \in 
\ZZ[\zeta]$ such that~$\sigma(\alpha) = \alpha^{2^{m-1}} + 2 \beta$. 
Squaring both sides yields
\[
	{\sigma(\alpha)}^2
	= {(\alpha^{2^{m-1}} + 2 \beta)}^2
	= \alpha^{2^{m}} + 4( \alpha^{2^{m-1}} \beta + \beta^2 )
\]
from which $\sigma(\alpha)^2-\alpha^{2^m} \in 4\ZZ[\zeta]$.
\end{proof}

\begin{lemma}\label{lemma:poly-cong-lemma3}
Let~$(\sigma, m)$ be as in~\Cref{prop:congruences}. Let $A \in 
\ZZ[\zeta][T]$ be monic of degree $D$ and denote by~$\alpha_1, \dots, 
\alpha_D$ its roots. Then
\[
	\sigma(e_i(\alpha_1^2, \dots, \alpha_D^2)) 
	\equiv e_i(\alpha_1^{2^m}, \dots, \alpha_D^{2^m}) \mod 4\ZZ[\zeta]
\]
for all $i \in \{1, \dots, D\}$.
\end{lemma}
\begin{proof}
We first claim that $\Psi_i(\alpha_1, \dots, \alpha_D) \in \ZZ[\zeta]$ 
for all $i \in \{1, \dots, D\}$. Since~$\Psi_i \in \ZZ[T_1, \dots, T_D]$ 
is symmetric, the fundamental theorem of symmetric polynomials 
implies that there exists $P_i \in \ZZ[T_1, \dots, T_D]$ such that 
$\Psi_i = P_i(e_1, \dots, e_D)$. Since $e_i(\alpha_1, \dots, 
\alpha_D) \in \ZZ[\zeta]$ for all $i \in \{1, \dots, D\}$ and 
$P_i \in \ZZ[T_1, \dots, T_D]$, we conclude that $\Psi_i(\alpha_1, 
\dots, \alpha_D) \in \ZZ[\zeta]$, proving the~claim.

Applying~\eqref{equation:numerical-congruence} to $\alpha = 
\Psi_i(\alpha_1, \dots, \alpha_D) \in \ZZ[\zeta]$ yields
\begin{equation}\label{equation:congruence-proof-lemma}
	\sigma\left(\Psi_i(\alpha_1, \dots, \alpha_D)\right) 
	\equiv {\Psi_i(\alpha_1, \dots, \alpha_D)}^{2^{m-1}} 
	\mod 2 \ZZ[\zeta].
\end{equation}

After multiplying both sides of~\eqref{equation:congruence-proof-lemma} 
by~$2$, use the definition of $\Psi_i$ and~\Cref{lemma:poly-cong-lemma1} 
on the right-hand side to obtain
\begin{multline}\label{eqn:eqn5}
	{\sigma (e_i(\alpha_1, \dots, \alpha_D))}^2
	- \sigma (e_i(\alpha_1^2, \dots, \alpha_D^2))
	\equiv\\ {e_i(\alpha_1, \dots, \alpha_D)}^{2^m}
	- e_i(\alpha_1^{2^m}, \dots, \alpha_D^{2^m})
	\mod 4\ZZ[\zeta].
\end{multline}
\Cref{lemma:poly-cong-lemma2} applied to $\alpha = e_i(\alpha_1, \dots, 
\alpha_D)$ yields 
\[
	\sigma(e_i(\alpha_1,\dots,\alpha_D))^2
	\equiv {e_i(\alpha_1, \dots, \alpha_D)}^{2^m} \mod 4\ZZ[\zeta].
\]	
The proof follows after substituting this last congruence into~\eqref{eqn:eqn5}.
\end{proof}

\begin{proof}[Proof of~\Cref{prop:congruences}]
Let $\alpha_1, \dots, \alpha_D$ be the roots of $A$. Recall that we 
want to show that
\[
	\sigma (e_i \left( f(\alpha_1), \dots, f(\alpha_D) \right))
	\equiv e_i \left( f^m(\alpha_1), \dots, f^m(\alpha_D) \right) 
	\mod 4\ZZ[\zeta]
\]
for all~$i \in \{ 1, \dots, D \}$.

The polynomial~$e_i \left(f(T_1), \dots, f(T_D) \right)$ is a 
symmetric polynomial in the variables $T_1^2, \dots, T_D^2$. 
Therefore, by the fundamental theorem of symmetric polynomials, there 
exists~$P_i \in \ZZ[T_1, \dots, T_D]$ such that 
\begin{equation}\label{equation:eq-prop}
	e_i \left( f(T_1), \dots, f(T_D) \right)
	= P_i\left( e_1(T_1^2, \dots, T_D^2), \dots, e_D(T_1^2, \dots, T_D^2) \right).
\end{equation}
We evaluate~\eqref{equation:eq-prop} at~$(\alpha_1,\dots,\alpha_D)$ 
and apply~$\sigma$ to both sides to obtain
\begin{multline*}
	\sigma (e_i\left( f(\alpha_1), \dots, f(\alpha_D)\right))
	= \sigma (P_i(e_1(\alpha_1^2, \dots, \alpha_D^2), \dots, e_D(\alpha_1^2, \dots, \alpha_D^2)))\\
	= P_i( \sigma (e_1(\alpha_1^2, \dots, \alpha_D^2)), \dots, \sigma (e_D(\alpha_1^2, \dots, \alpha_D^2))).
\end{multline*}
\Cref{lemma:poly-cong-lemma3} applied to $\alpha_1, \dots, \alpha_D$ 
implies that 
\begin{multline}\label{equation:eq-prop-2}
	P_i(\sigma(e_1(\alpha_1^2, \dots, \alpha_D^2)), \dots, \sigma(e_D(\alpha_1^2, \dots, \alpha_D^2)))\\
	\equiv P_i(e_1(\alpha_1^{2^m}, \dots, \alpha_D^{2^m}), \dots, e_D(\alpha_1^{2^m}, \dots, \alpha_D^{2^m})) \mod 4\ZZ[\zeta].
\end{multline}
By definition of~$P_i$, the right-hand side of~\eqref{equation:eq-prop-2} 
is equal to $e_i(f(\alpha_1^{2^{m - 1}}), \dots, f(\alpha_D^{2^{m - 1}}))$.

If $2$ is ramified in $\QQ(\zeta)$, then $m = 1$ and the proof 
follows. If $2$ is unramified in $\QQ(\zeta)$, then $m = 3$. The 
congruence
\[
	f^3(T) = T^8 - 4 T^6 + 4 T^4 - 1 \equiv f(T^4) \mod 4\ZZ[T]
\]
implies 
\[
	e_i(f^3(T_1), \dots, f^3(T_D))
	\equiv e_i(f(T_1^4), \dots, f(T_D^4)) \mod 4 \ZZ[T_1, \dots, T_D].
\]
Let $G \in \ZZ[T_1, \dots, T_D]$ be such that 
\[
	e_i(f^3(T_1), \dots, f^3(T_D)) - e_i(f(T_1^4), \dots, f(T_D^4)) 
	= 4 G(T_1, \dots, T_D).
\]
The polynomial $G$ is symmetric with integer coefficients, so 
$G(\alpha_1, \dots, \alpha_D) \in \ZZ[\zeta]$. Hence,
\[
	e_i(f^3(\alpha_1), \dots, f^3(\alpha_D))
	\equiv e_i(f(\alpha_1^4), \dots, f(\alpha_D^4)) \mod 4 \ZZ[\zeta]
\]
and the proof follows.
\end{proof}

If $2$ is ramified in $\QQ(\zeta^4)$, then the extension 
$\QQ(\zeta) / \QQ(\zeta^4)$ is quartic and cyclic by~\Cref{prop:cyclotomic-fields}.(iv.).
Let $\rho \in \Gal(\QQ(\zeta) / \QQ(\zeta^4))$ be a generator. Note 
that $\rho^2$ generates $\Gal(\QQ(\zeta)/\QQ(\zeta^2))$, and 
that~$\rho|_{\QQ(\zeta^2)}$ generates 
$\Gal(\QQ(\zeta^2) / \QQ(\zeta^4))$. Moreover, it holds that 
$\rho(\zeta^2) = - \zeta^2$. Indeed, since~$\rho(\zeta^4) = \zeta^4$, 
it follows that $\rho(\zeta^2) = \pm \zeta^2$ and since $\zeta^2 
\notin \QQ(\zeta^4)$ we conclude that $\rho(\zeta^2) = -\zeta^2$. 
Applying~\Cref{prop:congruences} to a monic polynomial $A \in 
\ZZ[\zeta^2][T]$ yields the following proposition, which we state for 
reference.

\begin{prop}\label{prop:rho-congruence}
If $A \in \ZZ[\zeta^2][T]$ is monic, then~$\rho A_1 \equiv A_1 
\mod 4 \ZZ[\zeta^2][T]$.
\end{prop}

\subsection{Power series}

We apply~\Cref{prop:congruences} to construct a power series 
in~$\ZZ[\zeta][\![1/T]\!]$. 

Consider the binomial series
\[
	\Phi_2(T) = \sum_{k = 0}^\infty \binom{1/2}{k} T^k 
	\in \QQ[\![T]\!]
\]
where
\[
	\binom{1/2}{k} 
	= \frac{1}{k!}\frac{1}{2}
	\left(\frac{1}{2} - 1\right)
	\dots
	\left(\frac{1}{2} - k + 1\right).
\]
The binomial series converges for all $z \in \CC$ with $|z| < 1$. 
Moreover ${\Phi_2(T)}^2 = 1 + T$ and~$\Phi_2(4 T)\in\ZZ[\![T]\!]$. 
The first identity is due to Newton, while the integrality property 
was proved by Dimitrov in~\cite{dimitrov}.

In the following proposition, we consider the ring of formal power 
series~$\ZZ[\zeta][\![1/T]\!]$. This ring is naturally endowed with 
the valuation~$v_\infty$ such that~$v_\infty(1/T) = 1$. A polynomial 
of degree~$D$ has valuation~$-D$ and therefore lies in the ring 
$\ZZ[\zeta](\!(1/T)\!)$ of formal Laurent series rather than in 
$\ZZ[\zeta][\![1/T]\!]$. Hence we will consider~$\ZZ[\zeta][T] 
\subset \ZZ[\zeta](\!(1/T)\!)$.

\begin{prop}\label{prop:construction-of-the-function}
Let~$A \in \ZZ[\zeta][T]$ be monic of degree $D$. If there 
exist~$(\sigma, m) \in \Gal(\QQ(\zeta) / \QQ) \times \NN$ and a root 
of unity $\xi \in \QQ(\zeta)$ such that $A_1 \in \ZZ[\xi][T]$ and 
$\sigma A_1 \equiv A_m \mod 4\ZZ[\xi][T]$, then there exists~$B \in 
\ZZ[\xi][T]$ such that
\begin{itemize}
	\item[(i.)] $\deg B\leq 2 D - 1$;
	\item[(ii.)]$\Phi_2\left(4 B/{\sigma A_1^2}\right)\in \ZZ[\xi]
	[\![1/T]\!]$;
	\item[(iii.)] $\Phi_2{\left({4 B}/{\sigma A_1^2}\right)}^2 = A_m/\sigma A_1$.
\end{itemize}
\end{prop}
\begin{proof}
By hypothesis, there exists~$B' \in \ZZ[\xi][T]$ such that~$\sigma 
A_1 = A_m - 4B'$. Multiplying by $\sigma A_1$ gives  
\begin{equation}\label{equation:equality}
	\sigma A_1^2 = A_m \cdot \sigma A_1 - 4 B,
\end{equation} 
where~$B = B' \cdot \sigma A_1 \in \ZZ[\xi][T]$. 
Equation~\eqref{equation:equality} rearranges as
\[
	\frac{A_m}{\sigma A_1} = 1 + \frac{4B}{\sigma A_1^2}.
\]
Since $\Phi_2(T)^2 = 1 + T$, it follows that 
\[
	\Phi_2\left(\frac{4 B}{\sigma A_1^2}\right)^2
	= \frac{A_m}{\sigma A_1},
\]
which proves (iii.).

We now prove (i.). If $B' = 0$, then (i.) is immediate. Assume that 
$B' \neq 0$. Since $\sigma A_1$ and $A_m$ are both monic of 
degree~$D$, their leading terms cancel, thus $\sigma A_1 - A_m$ has 
degree at most $D - 1$. Since $\sigma A_1 - A_m = - 4B'$, we have 
$\deg B' \le D - 1$. Therefore
\[
	\deg B = \deg B' + \deg \sigma A_1 \le D + (D - 1) = 2 D - 1,
\]
and (i.) follows. 

It remains to prove (ii.). Observe that for a monic~$C \in \ZZ[\xi][T]$ 
of degree~$E$, we have~$1/C \in \ZZ[\xi][\![1/T]\!]$. Indeed, if
$C(T) = T^E + a_1 T^{E - 1} + \cdots + a_E$, then
\begin{align}\label{eqn:exp-at-infty}
	\frac{1}{C}
	&= \frac{1}{T^E} \cdot \frac{1}{1 + \dfrac{a_1}{T} + \cdots + 
	\dfrac{a_E}{T^E}}\\\notag
	&= \frac{1}{T^E} \sum_{k = 0}^\infty (-1)^k{\left(\frac{a_1}{T} + 
	\cdots + \frac{a_E}{T^E}\right)}^k \in \ZZ[\xi][\![1/T]\!].
\end{align}
Applying this to $C = \sigma A_1^2$ and using (i.) we obtain 
$B/\sigma A_1^2 \in 1/T \ZZ[\xi][\![1/T]\!]$. Recall that $\Phi_2(4T) 
\in \ZZ[\![T]\!]$. Then 
\[
	\Phi_2\left(\frac{4B}{\sigma A_1^2}\right)
	= \sum_{k = 0}^\infty 
	\binom{1/2}{k} \cdot {\left(\frac{4B}{\sigma A_1^2}\right)}^k 
	\in \ZZ[\xi][\![1/T]\!].\qedhere
\]
\end{proof}

\begin{remark}\label{remark:convergence-disk}
The Laurent expansion of $1/C$ at infinity converges absolutely for 
$|z| > \max\left\{1, E\max_k|a_k|\right\}$. 
Indeed, 
\[
	\left| \frac{a_1}{z} + \cdots + \frac{a_E}{z^E} \right|
	\le \sum_{k = 1}^{E} \frac{|a_k|}{|z^k|}
	\le \frac{1}{|z|} E \max_{1 \le k \le E} |a_k| < 1.
\]
Now write $\sigma A_1(T) = T^D + 
a'_1 T^{D - 1} + \dots + a'_D$. Since $\sigma A_1$ is monic, the 
Laurent expansion of $1 / \sigma A_1$ at infinity converges for
$|z| > \max\left\{1, D\max_k|a'_k|\right\}$. Consequently 
$1 / \sigma A_1^2 = \sum_{k \ge 2D} c_k T^{-k}$ converges on the same 
domain. Moreover, since $\deg B \le 2D - 1$ 
by~\Cref{prop:construction-of-the-function}, we have 
\(
	 4 B / \sigma A_1^2 \in 1/T \ZZ[\xi][\![1/T]\!]
\). 
Thus $|4B / \sigma A_1^2| < 1$ for all $z$ with sufficiently large 
absolute value. Hence $\Phi_2(4B / \sigma A_1^2)$ converges for all 
$z$ with sufficiently large absolute value.
\end{remark}
\section{P\'olya--Bertrandias Theorem and Habegger--Schmidt bound}\label{sec:dimitrov-background}

In this section, we recall some results from potential theory and 
complex dynamics that will be used later to construct a suitable 
compact subset of $\CC$ and estimate its transfinite diameter.

Given a nonempty compact set $C \subset \CC$ and an integer $n \ge 2$, 
we define the \textit{n\emph{th} diameter} of $C$ by
\[
    \mathrm{d}_n(C) = \sup_{z_1, \dots, z_n \in C}
    \prod_{1 \le i < j \le n} | z_i - z_j |^{\frac{2}{n(n-1)}}.
\]  
The sequence ${(\mathrm{d}_n(C))}_{n \ge 2}$ is nonnegative and 
nonincreasing~\cite[see the proof of Theorem~5.5.2]{MR1334766}. 
It therefore converges to a limit, which we call the 
\textit{transfinite diameter} of $C$. It is convenient to set
$\mathrm{d}(\varnothing) = 0$.

Fekete and Szeg\H{o} proved that the transfinite diameter of a 
compact subset of $\CC$ equals its potential-theoretic 
capacity~\cite[Theorem 5.5.2]{MR1334766}. Estimates and formulas for 
the transfinite diameter of different compact subsets can be found 
in~\cite[Chapter~5.2]{MR1334766}.

We define the \textit{hedgehog generated by} $z_1, \dots, z_n \in 
\CC$ as the union of the segments $\{ t z_i\colon t \in [0,1] \}$. 
Define 
\[
	\mathtt{C}(f) = \max\left\{
		1, \sup_{z \in \CC} 
			\frac{\max\left\{|z|,1 \right\}}{e^{\hat\lambda_{f}(z)}}
			\right\}
\]
and denote the extended complex plane by $\widehat\CC = \CC \cup 
\{\infty\}$. We will use the following proposition.

\begin{prop}{\emph{\cite[Proposition 3.4]{MR4726503}}}\label{prop:habegger-schmidt}
Let $d\ge 2$ be an integer and let $f(T)=T^d+c\in \CC[T]$. 
Suppose that there exists $n\in\NN$ such that the set $\left\{f^{i}(0)
\colon 1 \le i \le n\right\}$ is contained in a hedgehog generated by 
at most $q \ge 1$ complex numbers. 
Let $m\ge 1$ and let $\alpha_1,\dots,\alpha_m\in\CC$. Then there 
exists a simply connected domain $U\subseteq \widehat{\CC}$ such that:
\begin{itemize}
	\item[(i.)] The set $\widehat{\CC}\setminus U$ is compact in 
	$\CC$ and contains $0,\alpha_1,\dots,\alpha_m$.
	\item[(ii.)] We have
	\[
		\mathrm{d}(\widehat{\CC}\setminus U)
		\le 4^{-\frac{1}{qd^n+m}}
		{\mathtt{C}(f)}^{\frac{1}{d^n}}
		\exp\max\{
			\hat\lambda_f(0),
			\hat\lambda_f(\alpha_1),
			\dots,
			\hat\lambda_f(\alpha_m)
		\}.
	\]
\end{itemize}
\end{prop}

\begin{remark}\label{remark:C(f)-estimate}
For $f(T) = T^2 - 1$, the forward orbit of $0$ is contained in~$\{-1, 0\}$ 
and therefore in the hedgehog generated by $-1$. Hence one may take 
$q = 1$. Moreover, Habegger and Schmidt proved that $\mathtt{C}(T^2 - 1) 
\le (1 + \sqrt{5})/2$; see~\cite[Lemmas~3.5 and~3.6]{MR4726503}. 
Finally, note that $\hat \lambda_f(0) = 0$ since $0$ is preperiodic.
\end{remark}

Let $K$ be a number field and $\Theta = \sum_{k = 0}^\infty a_k T^{- k} 
\in K [\![1/T]\!]$. For any field embedding $\tau \colon K 
\rightarrow \CC$, define $\tau(\Theta) = \sum_{k = 0}^\infty \tau(a_k)T^{-k}$.
The following special case of~\cite[Th\'eor\`eme~5.4.6]{MR0447195} 
will be used to derive a lower bound for the transfinite diameter.

\begin{thm}[P\'olya--Bertrandias]\label{thm:polya-bertrandias}
Let $K$ be a number field with ring of integers $O_K$. Consider a 
formal power series $\Theta\in O_K[\![1/T]\!]$ and suppose that for 
every field embedding $\tau \colon K \rightarrow \CC$ there exists a 
connected open subset $U_\tau\subseteq\CC$ such that the following 
properties hold:
\begin{itemize}
	\item[(i.)] the complement $\CC\setminus U_\tau$ is compact;
	\item[(ii.)] the formal power series $\tau(\Theta)$ converges for 
	all complex numbers $z$ with sufficiently large absolute value; and
	\item[(iii.)] the holomorphic function induced by $\tau(\Theta)$ 
	extends to a holomorphic function on $U_\tau$.
\end{itemize}
If $\prod_{\tau\colon K\hookrightarrow\CC}\mathrm{d}(\CC\setminus 
U_{\tau})<1$, then $\Theta$ is a rational function.
\end{thm}
\section{Proof of~\texorpdfstring{\Cref{thm:main-thm}}{Theorem~\ref*{thm:main-thm}}}\label{sec:proof}

Let $f(T) = T^2 - 1$. In this section, we give the proof of~\Cref{thm:main-thm}.

\begin{lemma}\label{lemma:coordinate-change}
Let~$P \in \CC[T]$ be a polynomial, and let~$R \subset \CC$ be a
subring. If $P\circ f \in R[T]$, then $P \in R[T]$.
\end{lemma}
\begin{proof}
If $P = 0$, the conclusion is immediate. For nonzero $P$, we proceed 
by induction on $D = \deg P \ge 0$. If $D = 0$, then $P \circ f = P$, 
so the conclusion follows from the hypothesis. Then, suppose that $D 
\ge 1$. Let us write $P(T) = a T^D + \tilde P(T)$ where $\tilde P \in 
\CC[T]$ with $\deg \tilde P < D$. By hypothesis, $B(T) = P(f(T))$ 
lies in $R[T]$. Comparing leading coefficients, we obtain $a \in R$. 
Thus, $B(T) - a{(T^2 - 1)}^D = \tilde P(T^2 - 1) \in R[T]$. As 
$\deg \tilde P < D$, we conclude $\tilde P \in R[T]$ by the induction 
hypothesis. This implies $P(T) = a T^D + \tilde P (T) \in R[T]$.
\end{proof}

\begin{lemma}\label{lemma:bound-epsilon}
For every~$D \ge 1$ there exists an integer~$n \ge 2$ such that 
\[
	\frac{1}{2^n}\left(
		\frac{\log 4}{1 + 2^{1 - n}D}
		- \log
			\frac{1 + \sqrt{5}}{2}
		\right)
	> \frac{1}{6D} \cdot \log\left(2^{1/5} \left( \sqrt{5} - 1 \right)\right)
	> 0.
\]
\end{lemma}
\begin{proof}
Take $n = \min\left\{m \in \NN \colon 2^m \ge 3D\right\} \ge 2$. From 
the minimality of $n$ it follows that $2^{n - 1} < 3 D \le 2^n$. 
Therefore $2^n < 6D$ and $2^{1 - n}D \le 2 / 3$. From the latter, it 
follows that
\[
	\frac{\log 4}{1 + 2^{1 - n}D}
	\ge \frac{\log 4}{1 + 2/3}
	= \frac{3}{5} \log 4.
\]
Note that both $2^{1/5}$ and $\sqrt{5} - 1$ are greater than $1$, 
so
\[
	\frac{3}{5}\log 4
		- \log
			\frac{1 + \sqrt{5}}{2}
	= \log(2^{1/5}(\sqrt{5} - 1))
	> 0.
\]
Therefore 
\[
	\frac{1}{2^n}\left(
		\frac{\log 4}{1 + 2^{1 - n}D}
		- \log
			\frac{1 + \sqrt{5}}{2}
		\right)
	> \frac{1}{6D} \left(
		\frac{3}{5}\log 4
		- \log
			\frac{1 + \sqrt{5}}{2}
		\right).\qedhere
\]
\end{proof}

Write $C' = \frac{1}{6} \cdot \log\left(2^{1/5} \left( \sqrt{5} - 1 
\right)\right)$. We are ready to prove the main technical theorem.

\begin{thm}\label{thm:general-thm}
Let $f(T) = T^2 - 1$. Let $K$ be a number field with ring of integers 
$O_K$. Let $P, Q \in O_K[T]$ be monic of degree $D \ge 1$. If there 
exists a nonrational $\Phi \in O_K[\![1/T]\!]$ such that 
$\Phi^2 = P / Q$, then
\[
	\frac{1}{[K \colon \QQ]}
	\sum_{\tau \colon K \rightarrow \CC}
	\left(
	\max_{\beta \colon \tau P(\beta) = 0} \hat \lambda_f(\beta)
	+ \max_{\beta \colon \tau Q(\beta) = 0} \hat \lambda_f(\beta)
	\right)
	\ge \frac{C'}{D}.
\]
\end{thm}
\begin{proof}
For each field embedding $\tau \colon K \rightarrow \CC$, write
\[
	\tau P = \prod_{k = 1}^D (T - \pi_{\tau, k})
\]
and 
\[
	\tau Q = \prod_{k = 1}^D (T - \upsilon_{\tau, k}).
\]

Since $P$ and $Q$ are monic of the same degree, the constant term of 
$P / Q$ at infinity is $1$. Thus the constant term of $\Phi$ is 
either $1$ or $-1$. Replacing $\Phi$ with~$-\Phi$, we may assume 
that its constant term is $1$.

We now check the hypotheses of~\Cref{thm:polya-bertrandias}. For each 
embedding $\tau$ we have
\begin{equation}\label{equation:eq-function-tau}
	\tau(\Phi)^2 
	= \frac{\tau P}{\tau Q} 
	= \prod_{k = 1}^D \frac{T - \pi_{\tau,k}}{T - \upsilon_{\tau,k}}.
\end{equation}
Write $\tau P / \tau Q = 1 + H_\tau$, where $H_\tau \in 1/T \CC[\![1/T]\!]$. 
Since $\tau P / \tau Q$ is rational and regular at infinity, the 
series $H_\tau$ converges for all sufficiently large $|z|$ and tends 
to zero as $z$ tends to infinity. Hence, 
$\tau(\Phi) = \Phi_2(H_\tau)$. Since $\Phi_2$ converges near $0$,
the series $\tau(\Phi)$ also converges for sufficiently large $|z|$. 

Choose $n$ as in~\Cref{lemma:bound-epsilon}. The set $\{f^i(0) 
\colon 1 \le i \le n\}$ is contained in the hedgehog $[-1, 0]$. For 
each embedding~$\tau \colon K \rightarrow \CC$, we 
apply~\Cref{prop:habegger-schmidt} with $d = 2$, $c = -1$, and 
$q = 1$ to the $2D$ complex numbers $\pi_{\tau, 1}, \dots, 
\pi_{\tau, D}, \upsilon_{\tau, 1}, \dots, \upsilon_{\tau, D}$. This 
gives a simply connected open set $U_\tau \subset \widehat \CC$ such 
that~$\pi_{\tau, 1}, \dots, \pi_{\tau, D}, \upsilon_{\tau, 1}, \dots, 
\upsilon_{\tau, D}, 0 \notin U_\tau$; the complement~$\widehat \CC 
\setminus U_\tau$ is compact in~$\CC$; and
\begin{equation}\label{equation:capacity-bound}
	\mathrm{d}\left(\widehat\CC\setminus U_\tau\right)
	\le 4^{-\frac{1}{2^n + 2 D}} {\left(\frac{1 + \sqrt{5}}{2}\right)}^{2^{-n}} 
	\exp\left(
		\max_{\beta\colon \tau (P Q)(\beta) = 0}
		\hat\lambda_f(\beta)
	\right).
\end{equation}
Here we used that $\hat \lambda_f(0) = 0$ and the estimate 
for~$\mathtt{C}(f)$ from~\Cref{remark:C(f)-estimate}. 

The rational function $\tau P / \tau Q$ is holomorphic and 
nonvanishing on $U_\tau$, with value $1$ at infinity. Since $U_\tau$ 
is simply connected, this function admits a holomorphic square root 
on $U_\tau$. Choosing the sign appropriately, this square root agrees 
with $\tau(\Phi)$ near infinity and therefore gives a holomorphic 
extension of $\tau(\Phi)$ to $U_\tau$.

Since $\Phi$ is nonrational,~\Cref{thm:polya-bertrandias} implies that
\begin{equation}\label{equation:computation1}
	0 \le \sum_{\tau \colon K \rightarrow \CC}
	\log \mathrm{d}\left(
		\widehat \CC \setminus U_\tau
		\right).
\end{equation}
Taking logarithms in~\eqref{equation:capacity-bound} and summing over 
the embeddings gives
\begin{multline}\label{equation:computation2}
	\sum_{\tau \colon K \rightarrow \CC}
	\log\mathrm{d}\left(\widehat\CC \setminus U_\tau\right)\\
	\le \sum_{\tau \colon K \rightarrow \CC}
	\left(
		- \frac{\log 4}{2^n + 2 D} 
		+ \frac{1}{2^n} \log \left(\frac{1+\sqrt{5}}{2}\right)
		+ \max_{\beta \colon \tau (P Q)(\beta) = 0}
		\hat\lambda_f\left(\beta\right)
	\right).
\end{multline}
Note that 
\begin{equation*}
	\max_{\beta \colon \tau (P Q)(\beta) = 0}
	\hat\lambda_f(\beta) 
	\le \max_{\beta\colon \tau P(\beta) = 0}
	\hat\lambda_f(\beta)
	+ \max_{\beta\colon \tau Q(\beta) = 0}
	\hat\lambda_f(\beta).
\end{equation*}

Combining~\eqref{equation:computation1} and~\eqref{equation:computation2}
with this inequality and dividing by $[K \colon \QQ]$, we obtain
\begin{multline*}
	\frac{1}{[K \colon \QQ]}
	\sum_{\tau \colon K \rightarrow \CC}
	\left(
		\max_{\beta \colon \tau P(\beta) = 0} \hat\lambda_f(\beta)
		+ \max_{\beta \colon \tau Q(\beta) = 0} \hat\lambda_f(\beta) 
	\right)
	\ge\\
	\frac{1}{2^n}
	\left(
		\frac{\log 4}{1 + 2^{1 - n} D} - \log \frac{1 + \sqrt{5}}{2}
	\right).
\end{multline*}
The proof follows since we choose $n$ as in~\Cref{lemma:bound-epsilon}.
\end{proof}

As a first application of~\Cref{thm:general-thm}, we establish a 
lower bound for the relative dynamical house under certain conditions 
on the auxiliary polynomials.

\begin{thm}\label{thm:easy-cases}
Let $f(T) = T^2 - 1$. Let $A \in \QQab[T]$ be monic, irreducible, of 
degree $D \ge 1$, and with wandering roots. Let $\zeta$ be a root of 
unity such that $A \in \ZZ[\zeta][T]$. If there exist $(\sigma, m) 
\in \Gal(\QQ(\zeta)/\QQ) \times \NN$ and a root of unity $\xi \in 
\QQ(\zeta)$ such that $A_1 \in \ZZ[\xi][T]$,  $\sigma A_1 
\equiv A_m \mod 4\ZZ[\xi][T]$ and $\sigma A_1 \neq A_m$, then one of 
the following conclusions holds.
\begin{itemize}
	\item[(i.)] The polynomial $A_1$ is not a square and 
	\[
		\log \house{A_{}}_f^\mathrm{ab}
        \ge \frac{C'}{(2 + 2^m) D}.
	\]
	\item[(ii.)] The polynomial $A_1$ is the square of an irreducible 
	polynomial in $\QQab[T]$.
	\end{itemize}
\end{thm}
\begin{proof}
On applying~\Cref{prop:construction-of-the-function} to~$A$ with 
$\sigma$, $m$ and $\xi$ as in the hypothesis, we get a power 
series~$\Phi\in\ZZ[\xi][\![1/T]\!]$ such that 
\[
	\Phi^2 = \frac{A_m}{\sigma A_1}
\]
with the stated properties.

We claim that $\Phi$ is rational if and only if $A_1$ is a square. If 
$A_1$ is a square, then also $A_m$ is a square and the claim follows. 
Vice versa, if $\Phi$ is rational, suppose, for a contradiction, that 
$A_1$ is not a square. \Cref{lemma:min-poly} implies that $A_1$ is 
irreducible over $\QQab$. It follows that $\sigma A_1$ is also 
irreducible. Write $\Phi = P / Q$ for $P, Q \in \QQ(\xi)[T]$ coprime. 
Then $P^2 \sigma A_1 = Q^2 A_m$. Note that $\sigma A_1$ does not 
divide $Q$. Indeed, if $Q = \sigma A_1 R$ for some $R \in \QQ(\xi)[T]$, 
then $P^2 = \sigma A_1 R^2 A_m$ and since $\sigma A_1$ is 
irreducible, it follows that $\sigma A_1$ divides $P$, contradicting 
the coprimality of $P$ and $Q$. We conclude that $\sigma A_1$ divides 
$A_m$. Since both polynomials are monic of degree $D$, this yields 
$\sigma A_1 = A_m$, contrary to the hypothesis.

We now distinguish two cases. 

\textsc{Case 1}: $\Phi$ is nonrational. In particular, $A_1$ is not a 
square. \Cref{lemma:min-poly} implies that $A_1$ is $\QQab$-irreducible. 
We apply~\Cref{thm:general-thm} to $\Phi$ with $K = \QQ(\xi)$, $P = A_m$, 
and $Q = \sigma A_1$. The above discussion on $\Phi$ guarantees that 
$\Phi \in \ZZ[\xi][\![1/T]\!]$. Hence 
\[
	\frac{1}{[\QQ(\xi) \colon \QQ]}
	\sum_{\tau \colon \QQ(\xi) \rightarrow \CC}
	\left(
	\max_{\tau A_m(\beta) = 0} \hat \lambda_f(\beta)
	+ \max_{\tau \sigma A_1(\beta) = 0} \hat \lambda_f(\beta)
	\right)
	\ge \frac{C'}{D}.
\]

Note that $\sigma(\QQ(\xi)) = \QQ(\xi)$. Composition with 
$\sigma|_{\QQ(\xi)}$ permutes the embeddings of $\QQ(\xi)$ into 
$\CC$, so
\[
	\sum_{\tau\colon\QQ(\xi)\rightarrow\CC} 
	\max_{\beta\colon \tau\sigma A_1(\beta) = 0}
	\hat\lambda_f(\beta) 
	= \sum_{\tau\colon\QQ(\xi)\rightarrow\CC} 
	\max_{\beta\colon \tau A_1(\beta) = 0}
	\hat\lambda_f(\beta).
\]
Since $\hat \lambda_f(f(z)) = 2 \hat \lambda_f(z)$ for all $z \in 
\CC$, it follows that 
\[
	\max_{\beta\colon \tau A_m(\beta) = 0}
	\hat\lambda_f(\beta) 
	= 2^{m - 1} \max_{\beta\colon \tau A_1(\beta) = 0}
	\hat\lambda_f(\beta).
\]
Hence we get
\[
	\frac{1}{[\QQ(\xi) \colon \QQ]}
	\sum_{\tau \colon \QQ(\xi)\rightarrow\CC}
	\max_{\beta \colon \tau A_1(\beta)=0}
	\hat\lambda_f(\beta) \ge
	\frac{1}{1 + 2^{m - 1}}
	\frac{C'}{D}.
\]

Note that~\Cref{lemma:functoriality-house} implies that 
$\log\house{A_1}_f^\mathrm{ab} = 2\log\house{A_{}}_f^\mathrm{ab}$.
Applying \Cref{lemma:2.1} proves (i.).

\textsc{Case 2}: $\Phi$ is a rational function. 
We proved that this is equivalent to~$A_1$ being a square. 
\Cref{lemma:min-poly} implies that $A_1$ is the square of an 
irreducible polynomial over $\QQab$. This proves (ii.).
\end{proof}

\subsection{Reduction}

To obtain the quantitative bound in~\Cref{thm:easy-cases}, we must 
verify both that $\sigma A_1 \neq A_m$ and that $A_1$ is not a 
square. The following reduction ensures the second condition.

Let $A \in \QQab[T]$ be a monic, irreducible polynomial of degree $D 
\ge 1$. Let $n \ge 0$ be the largest integer such that $A_n = B^{2^n}$ 
for some monic polynomial~$B$. Such an integer exists. Indeed, taking 
degrees gives $D = 2^n \deg B$, so~$2^n \le D$. We call the uniquely 
determined polynomial $B$ the \textit{reduction of A} and denote it 
by $\mathrm{Red}_A$. In particular, $\mathrm{Red}_A = A$ if $A_1$ is 
not a square.

\begin{prop}\label{prop:properties-of-the-reduction}
If $A \in \QQab[T]$ is irreducible and monic, then the 
following statements hold.
\begin{itemize}
	\item[(i.)] The polynomial $\mathrm{Red}_A \in \QQab[T]$ is 
	irreducible over $\QQab$.
	\item[(ii.)] If $A \in \ZZ[\zeta][T]$ for a root of unity 
	$\zeta$, then $\mathrm{Red}_A \in \ZZ[\zeta][T]$.
	\item[(iii.)] The polynomial $(\mathrm{Red}_A)_1$ is irreducible 
	over $\QQab$.
\end{itemize}
\end{prop}
\begin{proof}
Let $n \ge 0$ be the integer occurring in the definition of 
$\mathrm{Red}_A$, so that $A_n = (\mathrm{Red}_A)^{2^n}$. We claim 
that $A(T) = \mathrm{Red}_A(f^n(T))$. The zero set of 
${\mathrm{Red}_A}$ is $\{f^n(\alpha) \colon A(\alpha) = 0 \}$. Thus 
every root of $A$ is a root of $\mathrm{Red}_A(f^n(T))$. Since $A$ is 
separable and both polynomials are monic of the same degree, they are 
equal. This proves the claim.

Repeated application of~\Cref{lemma:coordinate-change} with $R = 
\QQab$ shows that $\mathrm{Red}_A \in \QQab[T]$. If $\mathrm{Red}_A = 
R_1 R_2$ for nonconstant $R_1, R_2 \in \QQab[T]$, then $A(T) = 
\mathrm{Red}_A(f^n(T)) = R_1(f^n(T)) R_2(f^n(T))$, contradicting the 
irreducibility of $A$. This proves~(i.). If moreover $A \in 
\ZZ[\zeta][T]$, then repeated application of~\Cref{lemma:coordinate-change}
with $R = \ZZ[\zeta]$ gives $\mathrm{Red}_A \in \ZZ[\zeta][T]$. This 
proves~(ii.).

It remains to prove~(iii.). Suppose, for a contradiction, that 
$(\mathrm{Red}_A)_1$ is reducible over $\QQab$. 
By~\Cref{lemma:min-poly} applied to $\mathrm{Red}_A$, we have 
$(\mathrm{Red}_A)_1 = C^2$ for some monic $C \in \QQab[T]$. Then
\[
	A_{n + 1}
	= (A_n)_1
    = (\mathrm{Red}_A^{2^n})_1
\]
where the first equality follows from~\Cref{lemma:prop-An}.(ii.) 
applied to $A$ with $m = 1$,~and 
\[
    (\mathrm{Red}_A^{2^n})_1
    = (\mathrm{Red}_A)_1^{2^n}
	= C^{2^{n + 1}},
\]
where the first equality follows from repeated application 
of~\Cref{lemma:prop-An}.(i.). In total $A_{n + 1} = C^{2^{n + 1}}$,
contradicting the maximality of $n$. Thus $(\mathrm{Red}_A)_1$ is 
irreducible over $\QQab$, proving~(iii.).
\end{proof}

\begin{lemma}\label{lemma:house-reduction}
Let $A \in \QQab[T]$ be a monic irreducible polynomial. Let $n \ge 0$ 
be the integer occurring in the definition of $\mathrm{Red}_A$. Then 
\[
	\log\house{\mathrm{Red}_A}_f^\mathrm{ab} 
	= 2^n \log\house{A_{}}_f^\mathrm{ab}.
\]
\end{lemma}
\begin{proof}
The zero sets of $\mathrm{Red}_A$ and $\mathrm{Red}_A^{2^n}$ are the 
same, so $\house{\mathrm{Red}_A}_f^\mathrm{ab} = 
\house{\mathrm{Red}_A^{{2^n}}}_f^\mathrm{ab}$. It follows that
\[
	\log \house{\mathrm{Red}_A}_f^\mathrm{ab} 
	= \log \house{\mathrm{Red}_A^{2^n}}_f^\mathrm{ab}
	= \log \house{A_n}_f^\mathrm{ab}
	= 2^n \log \house{A_{}}_f^\mathrm{ab},
\]
where the final equality follows from~\Cref{lemma:functoriality-house}.
\end{proof}

\begin{lemma}\label{lemma:fiber-reduction}
Let $A \in \QQab[T]$ be monic, irreducible, and of degree $D \ge 1$. 
Let $C \in \RR$ be such that 
$\log\house{\mathrm{Red}_A}_f^\mathrm{ab} \ge C/\deg\mathrm{Red}_A$. 
Then $\log\house{A_{}}_f^\mathrm{ab} \ge C/D$.
\end{lemma}
\begin{proof}
Let $n \ge 0$ be such that $A_n = \mathrm{Red}_A^{2^n}$. Taking 
degrees gives $D = 2^n \deg \mathrm{Red}_A$. 
\Cref{lemma:house-reduction} implies that $2^n \cdot 
\log\house{A_{}}_f^\mathrm{ab} = 
\log\house{\mathrm{Red}_A}_f^\mathrm{ab}$. 
Hence
\[
	\log\house{A_{}}_f^{\mathrm{ab}}
	= \frac{\log\house{\mathrm{Red}_A}_f^{\mathrm{ab}}}{2^n}
	\ge \frac{C}{2^n \cdot \deg \mathrm{Red}_A}
	= \frac{C}{D}.\qedhere
\]
\end{proof}

\subsection{More auxiliary lemmas}\label{subsection:degeneracy}

Recall that $\zeta$ denotes a root of unity. Let $A \in \ZZ[\zeta][T]$ 
be monic, irreducible, and of degree $D \ge 1$. Let $(\sigma, m)$ be 
as in~\Cref{prop:congruences}. 

The condition $\sigma A_1 \neq A_m$ in~\Cref{thm:easy-cases} depends 
on the ramification of $2$ in $\QQ(\zeta)$. We begin with the 
unramified case. In this case, we can take $m = 3$ and $\sigma = 
\Frob^2$.

\begin{lemma}\label{lemma:unramified-lemma}
Suppose that~$2$ is unramified in $\QQ(\zeta)$. Let~$A \in 
\ZZ[\zeta][T]$ be monic and irreducible over $\QQ(\zeta)$. 
If~$\Frob^2 A_1 = A_3$, then all roots of~$A$ are preperiodic.
\end{lemma}
\begin{proof}
Let~$\alpha$ be a root of~$A$. Choose $\tau \in 
\Gal(\overline\QQ / \QQ)$ extending $\Frob^2$. Since $f^3(\alpha)$ is 
a root of $A_3 = \Frob^2 A_1$, there exists a root $\alpha'$ of $A$ 
such that $\tau(f(\alpha')) = f^3(\alpha)$. Since $A$ is irreducible, 
there exists $\sigma \in \Gal(\overline\QQ / \QQ(\zeta))$ such that 
$\sigma(\alpha) = \alpha'$. Therefore
\[
	f^3(\alpha) = \tau(f(\alpha')) = \tau \sigma(f(\alpha)).
\]
By the Galois invariance and the functorial property of $\hat h_f$, 
we have that
\[
	8 \hat h_f(\alpha)
	= \hat h_f(f^{3}(\alpha))
	= \hat h_f(\tau \sigma(f(\alpha)))
	= 2 \hat h_f(\alpha),
\]
which implies~$\hat h_f (\alpha) = 0$. So $\alpha$ is preperiodic by 
the dynamical Kronecker~theorem.
\end{proof}

If $2$ is ramified in $\QQ(\zeta)$, then $m = 1$ and $\sigma = 
\kappa$ is the nontrivial element of $\Gal(\QQ(\zeta) / \QQ(\zeta^2))$. 
In this setting, the analog of~\Cref{lemma:unramified-lemma} is 
false; that is, there exists a monic irreducible polynomial with 
wandering roots such that $A \in \ZZ[\zeta][T]$ and $\kappa A_1 = 
A_1$. For example, consider $\zeta = i$ and $A(T) = T - i \in \ZZ[i][T]$. 
The element $i$ is wandering but $A_1(T) = T + 2 \in \ZZ[T]$.

\begin{lemma}\label{lemma:ramified-lemma1}
Let $A \in \QQ(\zeta)[T] \setminus \QQ(\zeta^2)[T]$ be monic, 
irreducible over $\QQ(\zeta)$, and of degree $D$. Then the following 
statements hold.
\begin{itemize}
    \item[(i.)] If $2$ is ramified in $\QQ(\zeta^4)$, then $A_1 
	\notin \QQ(\zeta^4)[T]$.
    \item[(ii.)] If $2$ is ramified in $\QQ(\zeta)$ and $A_1 \in 
	\QQ(\zeta^2)[T]$, then $\kappa A(T) = (-1)^D A(-T)$.
\end{itemize}
\end{lemma}
\begin{proof}
First we prove (i.). By~\Cref{prop:cyclotomic-fields}.(iv.), the 
extension $\QQ(\zeta) / \QQ(\zeta^4)$ is quartic cyclic. Let $\rho$ 
be a generator of $\Gal(\QQ(\zeta) / \QQ(\zeta^4))$. Note that 
$\rho^2 = \kappa$, where $\kappa$ is the nontrivial element of 
$\Gal(\QQ(\zeta) / \QQ(\zeta^2))$. Suppose, for a contradiction, that 
$A_1 \in \QQ(\zeta^4)[T]$. 
The polynomial $A_1(f(T))$ factors as~$A_1(f(T)) = (-1)^D A(T)A(-T)$. 
Since $A_1(f(T)) \in \QQ(\zeta^4)[T]$, it follows that
\[
    (-1)^D \rho A(T) \rho A(-T)
    = \rho A_1(f(T))
    = A_1(f(T))
    = (-1)^D A(T) A(-T).
\]
Since $A(T)$ and $(-1)^D A(-T)$ are monic and irreducible over 
$\QQ(\zeta)$, uniqueness of factorization implies that either $\rho 
A(T) = A(T)$ or $\rho A(T) = (-1)^D A(-T)$. In the first case, it 
follows that $A \in \QQ(\zeta^4)[T] \subseteq \QQ(\zeta^2)[T]$, a 
contradiction. In the second case, 
\[
    \kappa A(T) 
    = \rho^2 A(T)
    = (-1)^D \rho A(-T)
    = A(T).
\]
Since the fixed field of $\kappa$ is $\QQ(\zeta^2)$, this implies $A 
\in \QQ(\zeta^2)[T]$, a contradiction. This proves (i.).

As for point (ii.), we proceed similarly. Since $2$ is ramified in 
$\QQ(\zeta)$, the extension $\QQ(\zeta) / \QQ(\zeta^2)$ is quadratic 
by~\Cref{prop:cyclotomic-fields}.(iii.). Let $\kappa \in 
\Gal(\QQ(\zeta) / \QQ(\zeta^2))$ be the nontrivial element. Consider 
the polynomial $A_1(f(T))$. Then, since we assume that $A_1 \in 
\QQ(\zeta^2)[T]$, it follows that
\[
    (-1)^D \kappa A(T) \kappa A(-T) 
    = \kappa A_1(f(T))
    = A_1(f(T))
    = (-1)^D A(T) A(-T).
\]
Since $A$ and $(-1)^D A(-T)$ are monic and irreducible over 
$\QQ(\zeta)$, uniqueness of factorization implies that either 
$\kappa A = A$ or $\kappa A(T) = (-1)^D A(-T)$. The first case gives 
$A \in \QQ(\zeta^2)[T]$, a contradiction. The second case proves 
point~(ii.).
\end{proof}

If $A \in \ZZ[\zeta][T]$, we can characterize when $A_1 \in 
\ZZ[\zeta^2][T]$.

\begin{lemma}\label{lemma:ramified-lemma2}
Suppose that~$2$ is ramified in $\QQ(\zeta)$. Let~$A \in 
\ZZ[\zeta][T] \setminus \ZZ[\zeta^2][T]$ be monic, irreducible over 
$\QQ(\zeta)$, and of degree $D$. Then $A_1 \in \ZZ[\zeta^2][T]$ 
if and only if there exist $X, Y \in \ZZ[\zeta^2][T]$ of degree at 
most $D$ such that $X(-T) = (-1)^D X(T)$, $Y(-T) = (-1)^{D + 1} Y(T)$, 
and $A = X + \zeta Y$.
\end{lemma}
\begin{proof}
Write $A = X + \zeta Y$ for $X, Y \in \ZZ[\zeta^2][T]$. Suppose that 
$A_1 \in \ZZ[\zeta^2][T]$. Note that $\ZZ[\zeta] \cap \QQ(\zeta^2) = 
\ZZ[\zeta^2]$, so $A \notin \ZZ[\zeta^2][T]$ implies $A \notin 
\QQ(\zeta^2)[T]$. Then~\Cref{lemma:ramified-lemma1}.(ii.) gives 
$\kappa A(T) = (-1)^{D}A(-T)$. Recall that $\kappa(\zeta) = - \zeta$. 
It follows that
\[
	X(T) - \zeta Y(T) 
    = \kappa A(T) 
    = (-1)^D A(-T) 
    = (-1)^D X(-T) + \zeta (-1)^D Y(-T).
\]
Since $1$ and $\zeta$ are linearly independent over $\QQ(\zeta^2)$, 
we obtain the identities $X(T) = (-1)^D X(-T)$ and $-Y(T) = (-1)^D Y(-T)$. 

Conversely, suppose that there exist $X, Y \in \ZZ[\zeta^2][T]$ with 
the stated properties. Then
\[
	A_1(f(T))
	= A(T)\cdot (-1)^D A(-T)
	= X(T)^2 - \zeta^2 Y(T)^2 \in \ZZ[\zeta^2][T]
\] 
and the proof follows from~\Cref{lemma:coordinate-change} applied 
with $R = \ZZ[\zeta^2]$ and $P = A_1$.
\end{proof}

\begin{lemma}\label{lemma:ramified-lemma3}
Suppose that $2$ is ramified in $\QQ(\zeta)$. Let $A \in \ZZ[\zeta][T] 
\setminus \ZZ[\zeta^2][T]$ be monic, irreducible over $\QQ(\zeta)$, 
and of degree~$D$. If $A_1 \in \ZZ[\zeta^2][T]$, then there exist $P, 
Q \in \ZZ[\zeta^2][T]$ with $P$ monic of degree $\lfloor D/2 \rfloor$ 
such that
\[
    A_1(T) =
    \begin{cases}
        P(1 + T)^2 - \zeta^2(1 + T)Q(1 + T)^2 & \text{ if $D$ is even},\\
        (1 + T) P(1 + T)^2 - \zeta^2 Q(1 + T)^2 & \text{ if $D$ is odd}.
    \end{cases}
\]
\end{lemma}
\begin{proof}
Let $X, Y \in \ZZ[\zeta^2][T]$ be as in~\Cref{lemma:ramified-lemma2}. 
Note that since $A$ is monic, $X$ is monic of degree $D$. Write
\[
    A_1(f(T)) 
    = (-1)^D A(T) A(-T)
    = X(T)^2 - \zeta^2 Y(T)^2.
\]
Suppose that $D$ is even. Then $X$ is even and $Y$ is odd. Write 
$X(T) = P(T^2)$ and $Y(T) = T Q(T^2)$ for some $P, Q \in \ZZ[\zeta^2][T]$.
The polynomial $P$ is monic since $X$ is. We obtain the identity
\[
    A_1(f(T)) 
    = P(T^2)^2 - \zeta^2 T^2 Q(T^2)^2
\]
and the substitution $S = T^2 - 1$ proves the formula in the even 
case. 

The case where $D$ is odd is analogous and follows after writing 
$X(T) = TP(T^2)$ and $Y(T) = Q(T^2)$ for some $P, Q \in 
\ZZ[\zeta^2][T]$.
\end{proof}

\begin{lemma}\label{lemma:ramified-lemma4}
Suppose that~$2$ is ramified in $\QQ(\zeta^4)$. Let~$A \in 
\ZZ[\zeta][T] \setminus \ZZ[\zeta^2][T]$ be monic, irreducible over 
$\QQ(\zeta)$, and of degree $D$. Suppose that $A_1 \in \ZZ[\zeta^2][T]$ 
and that it is irreducible over $\QQ(\zeta^2)$. If $D$ is odd, then 
$A_2 \notin \ZZ[\zeta^4][T]$.
\end{lemma}
\begin{proof}
Suppose, for a contradiction, that $D$ is odd and that $A_2 \in 
\ZZ[\zeta^4][T]$. Denote the roots of $A$ by  $\alpha_1, \dots, 
\alpha_D$ and write 
\[
    A(T) 
    = \sum_{k = 0}^D 
    (-1)^k e_k(\alpha_1, \dots, \alpha_D) T^{D - k}
\]
and 
\[
    A_1(T) 
    = \sum_{k = 0}^D 
    (-1)^k e_k(f(\alpha_1), \dots, f(\alpha_D)) T^{D - k}
\]
where $e_0(T_1, \dots, T_D) = 1$, and $e_k$ is the $k$th elementary 
symmetric polynomial for $k \in \{1, \dots, D\}$. By 
\Cref{lemma:ramified-lemma1}.(ii.) it follows that $\kappa A(T) = 
- A(-T)$. Thus, we obtain $\kappa(e_1(\alpha_1, \dots, \alpha_D)) 
= - e_1(\alpha_1, \dots, \alpha_D)$, from which we conclude that 
there exists $\beta \in \ZZ[\zeta^2]$ such that $e_1(\alpha_1, 
\dots, \alpha_D) = \zeta \beta$. This follows by the same argument as 
in the proof of~\Cref{lemma:ramified-lemma2}. Moreover 
$\kappa (e_2(\alpha_1, \dots, \alpha_D)) = e_2(\alpha_1, \dots, 
\alpha_D)$, which implies that $e_2(\alpha_1, \dots, \alpha_D) 
\in~\ZZ[\zeta^2]$.

We claim that $\rho A_1(T) = - A_1(-T)$, where $\rho$ is a generator 
of the Galois group $\Gal(\QQ(\zeta) / \QQ(\zeta^4))$. Indeed, since 
$A_2 \in \ZZ[\zeta^4][T]$, it follows that $A_2(f(T)) \in \ZZ[\zeta^4][T]$. 
Since $A_2(f(T)) = - A_1(T) A_1(-T)$ and $A_1$ is irreducible over 
$\QQ(\zeta^2)$, it follows by uniqueness of factorization in 
$\QQ(\zeta^2)[T]$ that either $\rho A_1 = A_1$ or $\rho A_1(T) = 
-A_1(-T)$. The first conclusion implies that $A_1 \in \QQ(\zeta^4)[T]$, 
contrary to~\Cref{lemma:ramified-lemma1}.(i.). Thus the claim follows.

In particular, we get that 
\begin{equation}\label{eqn:1}
    \rho(e_1(f(\alpha_1), \dots, f(\alpha_D))) 
	= - e_1(f(\alpha_1), \dots, f(\alpha_D)).
\end{equation}
By using that $e_1(T_1, \dots, T_D) = T_1 + \dots + T_D$ together 
with Newton's identities, we obtain  
\begin{align}\label{eqn:2}
    e_1(f(\alpha_1), \dots, f(\alpha_D))
    &= e_1(\alpha_1^2, \dots, \alpha_D^2) - D\\
    &= e_1(\alpha_1, \dots, \alpha_D)^2 \notag
    - 2 e_2(\alpha_1, \dots, \alpha_D) - D\\
	&= \zeta^2\beta^2 -2 e_2(\alpha_1, \dots, \alpha_D) - D. \notag
\end{align}
Recall that $\rho(\zeta^2) = -\zeta^2$. Putting 
together~\eqref{eqn:1} and~\eqref{eqn:2} yields
\begin{equation*}
	-\zeta^2 \rho(\beta)^2 
	- 2 \rho(e_2(\alpha_1, \dots, \alpha_D)) 
	- D
	= -\zeta^2\beta^2 
	+ 2e_2(\alpha_1, \dots, \alpha_D) 
	+ D.
\end{equation*}

Note that $\rho(x) \equiv x \mod 2\ZZ[\zeta^2]$ for all $x \in 
\ZZ[\zeta^2]$. Let $\gamma, \delta \in \ZZ[\zeta^2]$ be such that 
$\rho(\beta) = \beta + 2 \gamma$ and $\rho(e_2(\alpha_1, \dots, 
\alpha_D)) = e_2(\alpha_1, \dots, \alpha_D) + 2\delta$. 
Substituting these equalities into the last displayed equation yields
\[
    -\zeta^2 (\beta + 2\gamma)^2 
	- 2(e_2(\alpha_1, \dots, \alpha_D) + 2 \delta) 
	- D
	= -\zeta^2\beta^2 
	+ 2e_2(\alpha_1, \dots, \alpha_D) 
	+ D.
\]

Rearranging the terms and dividing by $4$ gives
\[
      \zeta^2\beta\gamma 
    + \zeta^2\gamma^2 
    + e_2(\alpha_1, \dots, \alpha_D)
	+ \delta
    = - \frac{D}{2}.
\]
This is a contradiction since the left-hand side lies in $\ZZ[\zeta^2]$ 
while the right-hand side is not an algebraic integer since $D$ is 
odd.
\end{proof}

\begin{lemma}\label{lemma:ramified-case-hard}
Suppose that~$2$ is ramified in $\QQ(\zeta^4)$ and let $\rho \in 
\Gal(\QQ(\zeta) / \QQ(\zeta^4))$ be a generator. Let~$A \in 
\ZZ[\zeta][T] \setminus \ZZ[\zeta^2][T]$ be monic and irreducible 
over~$\QQ(\zeta)$. Suppose that $A_1 \in \ZZ[\zeta^2][T]$ and that it 
is irreducible over $\QQ(\zeta^2)$. If $A_2 \in \ZZ[\zeta^4][T]$, 
then $\rho A_1 \equiv A_1 \mod 4\ZZ[\zeta^2][T]$ and $\rho A_1 \neq 
A_1$.
\end{lemma}
\begin{proof}
Denote the degree of $A$ by $D$. \Cref{lemma:ramified-lemma4} 
implies that $D$ is even. \Cref{lemma:ramified-lemma3} gives two 
polynomials $P, Q \in \ZZ[\zeta^2][T]$ such that $P$ is monic of 
degree $D / 2$ and $A_1(T) = P(1 + T)^2 - \zeta^2(1 + T)Q(1 + T)^2$. 
Consider the identity $A_2(f(T)) = A_1(T) A_1(-T)$. From the 
hypothesis, it follows that~$\rho$ fixes $A_2$. Thus, since $A_1$ is 
irreducible in $\QQ(\zeta^2)[T]$, it follows that either $\rho A_1 = 
A_1$ or $\rho A_1(T) = A_1(-T)$. The first case implies that $A_1 \in 
\QQ(\zeta^4)[T]$, contrary to~\Cref{lemma:ramified-lemma1}.(i.).
Therefore $\rho A_1(T) = A_1(-T)$ and $\rho A_1 \neq A_1$.

It remains to show that $\rho A_1 \equiv A_1 \mod 4\ZZ[\zeta^2][T]$. 
Note that for all $B \in \ZZ[\zeta^2][T]$, it holds that $B(1 + T) 
\equiv B(1 - T) \mod 2\ZZ[\zeta^2][T]$. 
Moreover, since $P \in \ZZ[\zeta^2][T]$, it follows that $\rho P(1 + T) 
\equiv P(1 + T) \mod 2\ZZ[\zeta^2][T]$; cf.~\Cref{sec:polynomial-congruences}. 
Combining these two congruences gives $\rho P(1 + T) \equiv P(1 - T) 
\mod 2\ZZ[\zeta^2][T]$. It follows by squaring that $\rho P(1 + T)^2 
\equiv P(1 - T)^2 \mod 4\ZZ[\zeta^2][T]$. The same holds for $Q$. 
Recall that $\rho(\zeta^2) = - \zeta^2$. Thus we conclude that 
\[
    \rho A_1(T)
    \equiv P(1 - T)^2 + \zeta^2(1 + T)Q(1 - T)^2 
    \mod 4\ZZ[\zeta^2][T].
\]
Since $\rho A_1(T) = A_1(-T)$, it follows that
\[
    P(1 - T)^2 + \zeta^2(1 + T)Q(1 - T)^2 
    \equiv P(1 - T)^2 - \zeta^2(1 - T)Q(1 - T)^2
    \mod 4\ZZ[\zeta^2][T]
\]
from which we conclude that $2 \zeta^2 Q(1 - T)^2 \equiv 0 \mod 
4\ZZ[\zeta^2][T]$. 
Hence we obtain 
\begin{multline*}
    \rho A_1(T) - A_1(T) \\
    = P(1 - T)^2 - \zeta^2(1 - T)Q(1 - T)^2 
	- P(1 + T)^2 + \zeta^2(1 + T)Q(1 + T)^2 \\
    \equiv 2 \zeta^2 T Q(1 - T)^2 \mod 4\ZZ[\zeta^2][T].
\end{multline*}
Since $2 \zeta^2 Q(1 - T)^2 \equiv 0 \mod 4\ZZ[\zeta^2][T]$, we 
conclude that $\rho A_1 - A_1 \equiv 0 \mod 4\ZZ[\zeta^2][T]$.
This concludes the proof.
\end{proof}

\subsection{Proof of~\texorpdfstring{\Cref{thm:main-thm}}{Theorem~\ref*{thm:main-thm}}}

Recall that $C' = \frac{1}{6} \log(2^{1/5}(\sqrt{5} - 1))$. Let 
$\Frob$, $\kappa$, and $m$ be as in~\Cref{sec:polynomial-congruences}.

\begin{prop}\label{prop:easy-cases}
Let $A \in \QQab[T]$ be monic, irreducible over $\QQab$, and of 
degree $D$. Suppose that the roots of $A$ are wandering and that the 
coefficients of $A$ are algebraic integers. Let $\zeta$ be a root of 
unity such that $A \in \ZZ[\zeta][T]$. 
\begin{enumerate}
	\item[(i.)] If~$2$ is unramified in $\QQ(\zeta)$, then
	\[
		\log \house{A_{}}_f^\mathrm{ab}
		\ge \frac{C'}{10 D}.
	\]
	\item[(ii.)] If $2$ is ramified in $\QQ(\zeta)$, $A_1 \notin 
	\ZZ[\zeta^2][T]$, and $A_1$ is irreducible over $\QQ(\zeta)$, then 
	\[
		\log \house{A_{}}_f^\mathrm{ab}
		\ge \frac{C'}{4 D}.
	\]
\end{enumerate}
\end{prop}
\begin{proof}
Suppose that $2$ is unramified in $\QQ(\zeta)$. 
\Cref{prop:congruences} applied to $A$ with $(\sigma, m) = (\Frob^2, 3)$ 
implies that $\Frob^2 A_1 \equiv A_3 \mod 4 \ZZ[\zeta][T]$. As we are 
assuming that the roots of $A$ are wandering, \Cref{lemma:unramified-lemma} 
implies that $\Frob^2 A_1 \neq A_3$. 

If $A_1$ is irreducible over $\QQ(\zeta)$, then~\Cref{thm:easy-cases} 
applied to $A$ with $(\sigma, m) = (\Frob^2, 3)$ and $\xi = \zeta$ 
yields
\[
    \log \house{A}_f^\mathrm{ab}
    \ge \frac{C'}{10 D}.
\]

If $A_1$ is reducible over $\QQ(\zeta)$, it must be a square 
by~\Cref{lemma:min-poly}. The roots of the polynomial $R = 
\mathrm{Red}_A$ are wandering. By~\Cref{prop:properties-of-the-reduction},  
$R$ is monic, irreducible over~$\QQab$, and lies in $\ZZ[\zeta][T]$. 
Moreover $R_1$ is irreducible over~$\QQab$. Applying 
\Cref{prop:congruences} and \Cref{lemma:unramified-lemma} to $R$ 
yields $\Frob^2 R_1 \equiv R_3 \mod 4\ZZ[\zeta][T]$ and $\Frob^2 R_1 
\neq R_3$. Consequently, \Cref{thm:easy-cases} applied to $R$ with 
$(\sigma, m) = (\Frob^2, 3)$, and $\xi = \zeta$ gives
\[
    \log \house{R}_f^\mathrm{ab}
    \ge \frac{C'}{10 \deg R}.
\]
The proof of (i.) follows after applying~\Cref{lemma:fiber-reduction}. 

To prove (ii.), suppose that $2$ is ramified in $\QQ(\zeta)$, that 
$A_1$ is irreducible over $\QQ(\zeta)$, and that $A_1 \notin 
\ZZ[\zeta^2][T]$. Then~\Cref{prop:congruences} applied to $A$ with 
$(\sigma, m) = (\kappa, 1)$ implies that $\kappa A_1 \equiv A_1 \mod 
4\ZZ[\zeta][T]$. 
Note that the hypothesis implies that $\kappa A_1 \neq A_1$. 
Therefore, \Cref{thm:easy-cases} applied to $A$ with $(\sigma, m) = 
(\kappa, 1)$ and $\xi = \zeta$ yields
\[
	\log\house{A_{}}_f^\mathrm{ab}
	\ge \frac{C'}{4 D}.
\]
This proves (ii.).
\end{proof}

Before proving~\Cref{thm:main-thm}, we establish the following 
auxiliary bound which depends on ramification. After replacing 
$\zeta$ with $-\zeta$ if necessary, we may assume that the $2$-adic 
valuation of the order of $\zeta$ is positive.

\begin{prop}\label{prop:estimate-n}
Let $A \in \QQab[T]$ be monic, irreducible over $\QQab$, and of 
degree $D$. Suppose that the roots of $A$ are wandering. Let $\zeta$ 
be a root of unity such that $A \in \ZZ[\zeta][T]$. Let $v \ge 1$ be 
the $2$-adic valuation of the order of~$\zeta$. Then
\[
	\log\house{A_{}}_f^\mathrm{ab}
	\ge \frac{C'}{2^v 5 D}.
\]
\end{prop}
\begin{proof}
Denote by $R$ the polynomial $\mathrm{Red}_A$. By 
\Cref{prop:properties-of-the-reduction}, $R \in \ZZ[\zeta][T]$, and 
both $R$ and $R_1$ are irreducible over $\QQab$. Moreover, the 
roots of $R$ are wandering because they are iterates of the roots of 
$A$.

We proceed by induction on $v \ge 1$. If $v = 1$, then $2$ is 
unramified in~$\QQ(\zeta)$. Therefore, \Cref{prop:easy-cases} implies
\[
	\log\house{A_{}}_f^\mathrm{ab}
	\ge \frac{C'}{10 D},
\]
proving the base case. 

Suppose now that $v \ge 2$. Then $2$ is ramified in $\QQ(\zeta)$, and 
we distinguish two cases.

\textsc{Case 1}: Suppose that $R_1 \notin \ZZ[\zeta^2][T]$. 
Then~\Cref{prop:easy-cases}.(ii.) applied to $R$ yields
\[
	\log\house{R}_f^\mathrm{ab}
	\ge \frac{C'}{4 \deg R}
	\ge \frac{C'}{2^{v} 5 \deg R}.
\]
The desired bound for $\house{A_{}}_f^\mathrm{ab}$ follows 
from~\Cref{lemma:fiber-reduction}.

\textsc{Case 2}: Suppose that $R_1 \in \ZZ[\zeta^2][T]$. Note that 
the $2$-adic valuation of the order of $\zeta^2$ is $v - 1$. 
Moreover, the roots of $R_1$ are wandering. By the induction 
hypothesis applied to $R_1$ and $v - 1$, we have
\[
	\log\house{R_1}_f^\mathrm{ab}
	\ge \frac{C'}{2^{v - 1} 5 \deg R_1}
	= \frac{C'}{2^{v - 1} 5 \deg R}.
\]
By \Cref{lemma:functoriality-house} applied to $R$ we have
\[
	\log\house{R_1}_f^\mathrm{ab}
	= 2 \log\house{R}_f^\mathrm{ab}.
\]
Hence 
\[
	\log\house{R}_f^\mathrm{ab}
	\ge \frac{C'}{2^{v} 5 \deg R}.
\]
Applying again~\Cref{lemma:fiber-reduction} yields
\[
	\log\house{A_{}}_f^\mathrm{ab}
	\ge \frac{C'}{2^v 5 D}.\qedhere
\]
\end{proof}

\begin{proof}[Proof of~\Cref{thm:main-thm}]
By~\Cref{lemma:fiber-reduction}, it suffices to prove that
\[
    \log\house{R_{}}_f^\mathrm{ab}
	\ge \frac{C'}{40 \deg R}
\]
for the monic polynomial $R = \mathrm{Red}_A$. 
\Cref{prop:properties-of-the-reduction} implies that $R \in 
\ZZ[\zeta][T]$, and that both $R$ and $R_1$ are irreducible in 
$\QQab[T]$. Note that $R$ has wandering roots, since they are 
iterates of the roots of $A$.

If $R_1 \notin \ZZ[\zeta^2][T]$,~\Cref{prop:easy-cases} yields the 
desired bound. We may therefore assume that $R_1 \in \ZZ[\zeta^2][T]$.

If $2$ is unramified in $\QQ(\zeta)$, \Cref{prop:easy-cases}.(i.) 
applied to $R$ yields the desired bound.

Suppose now that $2$ is ramified in $\QQ(\zeta)$. Repeatedly replace 
$\zeta$ with $\zeta^2$ as long as $2$ is ramified in $\QQ(\zeta)$ and 
$R \in \ZZ[\zeta^2][T]$. If this process reaches a root of unity
$\zeta$ for which $2$ is unramified in $\QQ(\zeta)$, then we are 
done. Otherwise, we may assume $R \in \ZZ[\zeta][T] \setminus 
\ZZ[\zeta^2][T]$. We distinguish two cases according to whether $2$ 
is ramified in $\QQ(\zeta^4)$.

\textsc{Case 1}: If $2$ is unramified in $\QQ(\zeta^4)$, then 
the $2$-adic valuation $v$ of the order of $\zeta$ is at most $3$. 
\Cref{prop:estimate-n} implies that
\[
	\log\house{R}_f^\mathrm{ab}
	\ge \frac{C'}{2^v 5 \deg R}
	\ge \frac{C'}{2^3 5 \deg R}
	= \frac{C'}{40 \deg R}.
\]

\textsc{Case 2}: If $2$ is ramified in $\QQ(\zeta^4)$, the extension 
$\QQ(\zeta) / \QQ(\zeta^4)$ is quartic cyclic by~\Cref{prop:cyclotomic-fields}. 
Let $\rho$ be a generator of $\Gal(\QQ(\zeta) / \QQ(\zeta^4))$. We 
distinguish three~subcases.

\textsc{Subcase 2a}: Suppose that $R_2 \in \ZZ[\zeta^4][T]$. It 
follows from \Cref{lemma:ramified-case-hard} that $\rho R_1 \neq R_1$ 
and $\rho R_1 \equiv R_1 \mod 4\ZZ[\zeta^2][T]$. Recall that we are 
assuming $R_1$ irreducible. Hence~\Cref{thm:easy-cases} applied to 
$R$ with $(\sigma, m) = (\rho, 1)$ and $\xi = \zeta^2$ yields 
\[
	\log\house{R_{}}_f^\mathrm{ab}
	\ge \frac{C'}{4 \deg R}.
\]

\textsc{Subcase 2b}: Suppose that $R_2 \notin \ZZ[\zeta^4][T]$ and 
$R_2$ is not a square. \Cref{prop:rho-congruence} applied to $R_1 \in 
\ZZ[\zeta^2][T]$ gives $\rho R_2 \equiv R_2 \mod 4 \ZZ[\zeta^2][T]$. 
Moreover, the assumption $R_2 \notin \ZZ[\zeta^4][T]$ implies 
$\rho R_2 \neq R_2$. Hence, \Cref{thm:easy-cases} applied to $R_1$ 
with $(\sigma, m) = (\rho|_{\QQ(\zeta^2)}, 1)$ and $\xi = \zeta^2$ 
gives
\[
	\log\house{R_{}}_f^\mathrm{ab}
	= \frac{1}{2} \log\house{R_1}_f^\mathrm{ab}
	\ge \frac{C'}{2 \cdot 4 \deg R}
	= \frac{C'}{8 \deg R},
\]
where the first equality follows from~\Cref{lemma:functoriality-house}.

\textsc{Subcase 2c}: Suppose that $R_2 \notin \ZZ[\zeta^4][T]$ and 
$R_2 = B^2$ for some monic $B \in \ZZ[\zeta^2][T]$. The polynomial 
$B$ is irreducible over $\QQab$ by~\Cref{lemma:min-poly} applied 
to~$R_1$. Note that~$R_1 = B \circ f$. Indeed, the polynomial 
$B \circ f$ has degree~$\deg R$ and vanishes on $f(\alpha)$ for all 
$\alpha \in \overline \QQ$ such that $R(\alpha) = 0$. Hence $R_1$ 
divides~$B \circ f$. Since both polynomials are monic of degree 
$\deg R$ and $R_1$ is irreducible, they are equal. In particular, 
$R_1$ is even; i.e.\ $R_1(T) = R_1(-T)$. 

We claim that $\rho B \neq B$ and $\rho B \equiv B \mod 
4\ZZ[\zeta^2][T]$. If the first part were false, then $R_2 = B^2$ 
would be fixed by $\rho$, a contradiction with $R_2 \notin 
\ZZ[\zeta^4][T]$. 

Since $R_2$ is a square, the degree $\deg R$ has to be even. 
 Therefore,~\Cref{lemma:ramified-lemma3} gives that
\[
	R_1(T) = P(1 + T)^2 - \zeta^2 (1 + T) Q(1 + T)^2
\]	
for some $P, Q \in \ZZ[\zeta^2][T]$. Note that the congruence 
$S(1 + T) \equiv S(1 - T) \mod 2 \ZZ[\zeta^2][T]$ holds for all 
$S \in \ZZ[\zeta^2][T]$ and implies by squaring that $S(1 + T)^2 
\equiv S(1 - T)^2 \mod 4\ZZ[\zeta^2][T]$. Therefore, since $R_1$ is 
even, it follows that
\begin{align*}
	0 
	&= R_1(T) - R_1(-T)\\
	&= P(1 + T)^2 - P(1 - T)^2 - \zeta^2 (1 + T) Q(1 + T)^2 + \zeta^2 (1 - T) Q(1 - T)^2\\
	&\equiv - 2\zeta^2 T Q(1 + T)^2 \mod 4\ZZ[\zeta^2][T].
\end{align*}
Since $\zeta^2$ is a unit, we conclude that $Q(1 + T)^2 \in 
2\ZZ[\zeta^2][T]$. 

We next show that $\rho R_1 \equiv R_1 \mod 4\ZZ[\zeta^2][T]$. Since 
for all $S \in \ZZ[\zeta^2][T]$ it holds that $\rho S \equiv S 
\mod 2 \ZZ[\zeta^2][T]$, it follows that $\rho P(1 + T)^2 \equiv 
P(1 + T)^2 \mod 4\ZZ[\zeta^2][T]$ and $\rho Q(1 + T)^2 \equiv 
Q(1 + T)^2 \mod 4 \ZZ[\zeta^2][T]$. Recall that $\rho(\zeta^2) = 
-\zeta^2$. So
\[
	\rho R_1(T) - R_1(T)
	\equiv 2 \zeta^2(1 + T)Q(1 + T)^2 \mod 4\ZZ[\zeta^2][T].
\]
Since $Q(1 + T)^2 \in 2\ZZ[\zeta^2][T]$, we conclude that $\rho R_1 
\equiv R_1 \mod 4\ZZ[\zeta^2][T]$. 

Now let $C = \rho B - B$. By the above discussion, the polynomial 
$C \circ f = \rho R_1 - R_1$ lies in $4\ZZ[\zeta^2][T]$. Hence $C \in 
4 \ZZ[\zeta^2][T]$ by~\Cref{lemma:coordinate-change} applied to $P = 
C / 4$ and the ring $\ZZ[\zeta^2]$. The claim $\rho B \neq B$ and 
$\rho B \equiv B \mod 4\ZZ[\zeta^2][T]$ is now proven.

Let $\tilde B \in \ZZ[\zeta^2][T]$ be such that $\rho B - B = 4 
\tilde B$. Since $B$ and~$\rho B$ are monic of degree $\deg R / 2$, 
we have $\deg \tilde B < \deg R / 2$. Then
\[
	\frac{B}{\rho B} = 1 - \frac{4 \tilde B}{\rho B},
\]
and $\tilde B / \rho B \in 1/T \ZZ[\zeta^2][\![1/T]\!]$. Define 
$\Theta = \Phi_2(- 4 \tilde B / \rho B)$. Since $\Phi_2(4 T) \in 
\ZZ[\![T]\!]$, it follows that $\Theta \in \ZZ[\zeta^2][\![1/T]\!]$. 
By construction $\Theta^2 = B / \rho B$. 

The series $\Theta$ is nonrational. Indeed, since $B$ and~$\rho B$ 
are distinct, monic and irreducible over $\QQab$, the ratio $B / \rho B$ 
has simple zeros and poles and therefore it is not the square of a 
rational function.  

Applying~\Cref{thm:general-thm} with $K = \QQ(\zeta^2)$, $P = B$, $Q 
= \rho B$, and $\Phi = \Theta$ yields
\[
	\frac{1}{[\QQ(\zeta^2) \colon \QQ]}
	\sum_{\tau \colon \QQ(\zeta^2) \rightarrow \CC}
	\left(
	\max_{\beta: \tau B(\beta) = 0} \hat \lambda_f(\beta)
	+ \max_{\beta: \tau \rho B(\beta) = 0} \hat \lambda_f(\beta)
	\right)
	\ge \frac{C'}{\deg B}
	= \frac{2 C'}{\deg R}.
\]
As $\tau$ varies, $\tau\rho$ ranges over all embeddings of 
$\QQ(\zeta^2)$ in $\CC$. Therefore,
\[
	\sum_{\tau \colon \QQ(\zeta^2) \rightarrow \CC}
	\left(
	\max_{\beta: \tau B(\beta) = 0} \hat \lambda_f(\beta)
	+ \max_{\beta: \tau \rho B(\beta) = 0} \hat \lambda_f(\beta)
	\right)
	= 2 \sum_{\tau \colon \QQ(\zeta^2) \rightarrow \CC}
	\max_{\beta: \tau B(\beta) = 0} \hat \lambda_f(\beta).
\]
Combining the last two displayed equations with~\Cref{lemma:2.1} 
yields
\[
	\log \house{B_{}}_f^\mathrm{ab} \ge \frac{C'}{\deg R}.
\]
We apply~\Cref{lemma:functoriality-house} to conclude that 
\[
	\log\house{R_{}}_f^\mathrm{ab}
	= \frac{1}{4}\ \log\house{R_2}_f^\mathrm{ab}
	= \frac{1}{4}\ \log\house{B^2}_f^\mathrm{ab}
	= \frac{1}{4}\ \log\house{B_{}}_f^\mathrm{ab}
	\ge \frac{1}{4} \cdot \frac{C'}{\deg R}
	= \frac{C'}{4 \deg R}.
\]

In every case,
\[
    \log\house{R_{}}_f^\mathrm{ab}
    \ge \frac{C'}{40 \deg R}.
\]
The desired bound for $A$ follows from~\Cref{lemma:fiber-reduction}.
\end{proof}

\begin{proof}[Proof of~\Cref{cor:height-bound}]
Write $C = {\log(2^{1/5}(\sqrt{5} -1))}/{240}$. Let $A \in \QQab[T]$ 
be the minimal polynomial of $\alpha$ over $\QQab$. Let $K$ be the 
smallest number field containing the coefficients of $A$. Then the 
minimal polynomial of $\alpha$ over~$\QQ$ is
\[
    N(T) = \prod_{\sigma \colon K \rightarrow \CC}
    \sigma A.
\]
Note that $\deg N = [K \colon \QQ] \deg A$. 
The Call--Silverman height of $\alpha$ can be written as
\begin{equation}\label{eqn:height-bound2}
    \hat h_f(\alpha)
    = \frac{1}{\deg N} \sum_{\beta \colon N(\beta) = 0} 
    \hat \lambda_f(\beta)
    = \frac{1}{\deg N} 
    \sum_{\sigma \colon K \rightarrow \CC}
	\sum_{\beta \colon \sigma A(\beta) = 0} 
    \hat \lambda_f(\beta).
\end{equation}
Bounding the inner sum in the right-hand side of~\eqref{eqn:height-bound2} 
below by its largest summand yields 
\[
    \hat h_f(\alpha)
    \ge \frac{1}{\deg N} \sum_{\sigma \colon K \rightarrow \CC} 
    \max_{\beta \colon \sigma A(\beta) = 0} 
    \hat \lambda_f(\beta).
\]
Applying~\Cref{thm:main-thm} gives
\begin{equation*}
    \hat h_f(\alpha)
    \ge \frac{1}{\deg N} \sum_{\sigma \colon K \rightarrow \CC} 
    \max_{\beta \colon \sigma A(\beta) = 0} 
    \hat \lambda_f(\beta)
    \ge \frac{1}{\deg N} \frac{[K \colon \QQ] C}{\deg A}
    = \frac{C}{(\deg A)^2}.
\end{equation*}
This concludes the proof.
\end{proof}
\section{Final remarks}

As a final remark, we formulate a relative version 
of~\Cref{conj:dyn-sz}. 

Consider $G_\QQ = \Gal(\overline\QQ / \QQ)$. Restriction gives a 
surjective homomorphism
\[
    G_\QQ \longrightarrow \Gal(\QQab / \QQ).
\]
The group $G_{\QQab} = \Gal(\overline\QQ / \QQab)$ is the kernel of 
this homomorphism, and thus is a normal subgroup of $G_\QQ$. 
Moreover, the quotient group $G_\QQ / G_{\QQab}$ is isomorphic to 
$G^\mathrm{ab}_\QQ = \Gal(\QQab/\QQ)$.

Let~$f \in \QQab[T]$ be a polynomial of degree at least $2$ and 
suppose that its coefficients are algebraic integers. We define an 
action of $G^\mathrm{ab}_\QQ$ on the set of pairs $(G_{\QQab} \cdot 
\alpha, g)$, where $G_{\QQab} \cdot \alpha = \{\sigma(\alpha) \colon 
\sigma \in G_{\QQab}\}$ and $g$ ranges over the Galois conjugates of 
$f$. For each $\tau \in G^\mathrm{ab}_\QQ$, let $\tilde \tau$ be a 
lift of $\tau$ to~$G_\QQ$. Then define 
\[
    \tau \cdot (G_{\QQab} \cdot \alpha, g) 
    = (\tilde\tau(G_{\QQab} \cdot \alpha), \tilde\tau(g)), 
\]    
where
\(
    \tilde\tau(G_{\QQab} \cdot \alpha)
    = \{\tilde\tau(\beta) \colon \beta \in G_{\QQab} \cdot \alpha\}.
\)
This action is well defined. Indeed, if $\tilde \tau_1$ and $\tilde \tau_2$ 
are two lifts of $\tau$ to $G_\QQ$, then $\tilde \tau_2 = 
\tilde \tau_1 \rho$ for some $\rho \in G_{\QQab}$. Hence 
\[
    \tilde \tau_2 (G_{\QQab} \cdot \alpha)
    = \tilde \tau_1 \rho(G_{\QQab} \cdot \alpha)
    = \tilde \tau_1 (G_{\QQab} \cdot \alpha)
\]
and $\tilde \tau_2(g) = \tilde \tau_1 \rho(g) = \tilde\tau_1(g)$.
Let $\Gamma_f(\alpha)$ denote the orbit of $(G_{\QQab} \cdot \alpha, f)$ 
under the action of $G^\mathrm{ab}_\QQ$.
\begin{conj}[Relative Dynamical Schinzel--Zassenhaus]\label{conj:dyn-rel-sz}
Let $f \in \QQab[T]$ be a polynomial of degree at least $2$ whose 
coefficients are algebraic integers. Then, there exists $C = C(f) > 
0$ such that
\begin{equation}\label{eqn:rel-dyn-sz}
    \frac{1}{|\Gamma_f(\alpha)|}
    \sum_{(\mathcal{O},g) \in \Gamma_f(\alpha)}
    \max_{\beta \in \mathcal{O}} \hat\lambda_{g}(\beta)
    \ge \frac{C}{[\QQab(\alpha) \colon \QQab]}
\end{equation}
for all $f$-wandering algebraic integers $\alpha \in \overline\QQ$.
\end{conj}
Dimitrov~\cite[Theorem 2]{dimitrov} proved~\Cref{conj:dyn-rel-sz}
for $f(T) = T^2$. The case $f(T) = T^2 - 2$ follows from Dimitrov's 
result. In this paper, we proved it for $f(T) = T^2 - 1$.

\defbibheading{mybibheading}{
  \section*{References}
}
\printbibliography[heading = mybibheading]

@book{MR0447195,
  author     = {Amice, Y.},
  title      = {Les nombres {$p$}-adiques},
  series     = {Collection SUP: ``Le Math\'{e}maticien''},
  volume     = {14},
  publisher  = {Presses Universitaires de France, Paris},
  year       = {1975},
  pages      = {190},
  mrclass    = {12B40 (46B99)},
  mrnumber   = {447195},
  mrreviewer = {W.\ Bartenwerfer}
}

@article{MR1740514,
  author     = {Amoroso, F. and Dvornicich, R.},
  title      = {A lower bound for the height in abelian extensions},
  journal    = {J. Number Theory},
  fjournal   = {Journal of Number Theory},
  volume     = {80},
  year       = {2000},
  number     = {2},
  pages      = {260--272},
  mrclass    = {11R18 (11G50)},
  mrnumber   = {1740514},
  mrreviewer = {Christopher\ Smyth},
  doi        = {10.1006/jnth.1999.2451},
  url        = {https://doi.org/10.1006/jnth.1999.2451}
}

@article{MR1817715,
  author   = {Amoroso, F. and Zannier, U.},
  title    = {A relative {D}obrowolski lower bound over abelian extensions},
  journal  = {Ann. Scuola Norm. Sup. Pisa Cl. Sci. (4)},
  fjournal = {Annali della Scuola Normale Superiore di Pisa. Classe di
              Scienze. Serie IV},
  volume   = {29},
  year     = {2000},
  number   = {3},
  pages    = {711--727},
  mrclass  = {11G50 (11J99 11R47)},
  mrnumber = {1817715},
  url      = {http://www.numdam.org/item?id=ASNSP_2000_4_29_3_711_0}
}

@article{MR2651944,
  author     = {Amoroso, F. and Zannier, U.},
  title      = {A uniform relative {D}obrowolski's lower bound over abelian
                extensions},
  journal    = {Bull. Lond. Math. Soc.},
  fjournal   = {Bulletin of the London Mathematical Society},
  volume     = {42},
  year       = {2010},
  number     = {3},
  pages      = {489--498},
  mrclass    = {11R47 (11G50)},
  mrnumber   = {2651944},
  mrreviewer = {Christopher\ Smyth},
  doi        = {10.1112/blms/bdq008},
  url        = {https://doi.org/10.1112/blms/bdq008}
}

@article {MR790657,
    AUTHOR = {Boyd, David W.},
     TITLE = {The maximal modulus of an algebraic integer},
   JOURNAL = {Math. Comp.},
  FJOURNAL = {Mathematics of Computation},
    VOLUME = {45},
      YEAR = {1985},
    NUMBER = {171},
     PAGES = {243--249, S17--S20},
      ISSN = {0025-5718,1088-6842},
   MRCLASS = {11R06},
  MRNUMBER = {790657},
MRREVIEWER = {Christopher\ Smyth},
       DOI = {10.2307/2008062},
       URL = {https://doi.org/10.2307/2008062},
}

@article{dimitrov,
  title         = {The Schinzel-Zassenhaus conjecture},
  author        = {Dimitrov, V.},
  year          = {2026},
  note          = {To appear.},
  fjournal      = {Annals of Mathematics},
  journal       = {Ann. of Math.},
}

@article {MR1481371,
    AUTHOR = {Dubickas, A.},
     TITLE = {The maximal conjugate of a non-reciprocal algebraic integer},
   JOURNAL = {Liet. Mat. Rink.},
  FJOURNAL = {Matematikos ir Informatikos Institutas. Lietuvos
              Matematik\polhk u\ Draugija. Vilniaus Universitetas. Lietuvos
              Matematikos Rinkinys},
    VOLUME = {37},
      YEAR = {1997},
    NUMBER = {2},
     PAGES = {168--174},
      ISSN = {0132-2818},
   MRCLASS = {11R06},
  MRNUMBER = {1481371},
MRREVIEWER = {D.\ Boyd},
       DOI = {10.1007/BF02465885},
       URL = {https://doi.org/10.1007/BF02465885},
}

@article{MR4726503,
  author   = {Habegger, P. and Schmidt, H.},
  title    = {Lower {B}ounds for the {C}anonical {H}eight of a {U}nicritical
              {P}olynomial and {C}apacity},
  journal  = {Forum Math. Sigma},
  fjournal = {Forum of Mathematics. Sigma},
  volume   = {12},
  year     = {2024},
  pages    = {Paper No. e45},
  mrclass  = {11G50 (14G40 30C85 37P15)},
  mrnumber = {4726503},
  doi      = {10.1017/fms.2023.112},
  url      = {https://doi.org/10.1017/fms.2023.112}
}

@book {MR2193309,
    AUTHOR = {Milnor, John},
     TITLE = {Dynamics in one complex variable},
    SERIES = {Annals of Mathematics Studies},
    VOLUME = {160},
   EDITION = {Third},
 PUBLISHER = {Princeton University Press, Princeton, NJ},
      YEAR = {2006},
     PAGES = {viii+304},
      ISBN = {978-0-691-12488-9; 0-691-12488-4},
   MRCLASS = {37Fxx (30-01 30D05 37-01)},
  MRNUMBER = {2193309},
}

@book{MR1697859,
  author     = {Neukirch, J.},
  title      = {Algebraic number theory},
  series     = {Grundlehren der mathematischen Wissenschaften [Fundamental
                Principles of Mathematical Sciences]},
  volume     = {322},
  publisher  = {Springer-Verlag, Berlin},
  year       = {1999},
  pages      = {xviii+571},
  isbn       = {3-540-65399-6},
  mrclass    = {11Rxx (11-02 11S15 11S31 14C40)},
  mrnumber   = {1697859},
  mrreviewer = {Cornelius\ Greither},
  doi        = {10.1007/978-3-662-03983-0},
  url        = {https://doi.org/10.1007/978-3-662-03983-0}
}

@article {MR4395356,
    AUTHOR = {Pritsker, Igor E.},
     TITLE = {House of algebraic integers symmetric about the unit circle},
   JOURNAL = {J. Number Theory},
  FJOURNAL = {Journal of Number Theory},
    VOLUME = {236},
      YEAR = {2022},
     PAGES = {388--403},
      ISSN = {0022-314X,1096-1658},
   MRCLASS = {11R06 (12D10 30C10 30C15)},
  MRNUMBER = {4395356},
MRREVIEWER = {E.\ J.\ Barbeau},
       DOI = {10.1016/j.jnt.2021.07.026},
       URL = {https://doi.org/10.1016/j.jnt.2021.07.026},
}

@book{MR1334766,
  author     = {Ransford, T.},
  title      = {Potential theory in the complex plane},
  series     = {London Mathematical Society Student Texts},
  volume     = {28},
  publisher  = {Cambridge University Press, Cambridge},
  year       = {1995},
  pages      = {x+232},
  isbn       = {0-521-46120-0; 0-521-46654-7},
  mrclass    = {31-02},
  mrnumber   = {1334766},
  mrreviewer = {D.\ H.\ Armitage},
  doi        = {10.1017/CBO9780511623776},
  url        = {https://doi.org/10.1017/CBO9780511623776}
}

@book{MR2316407,
  author     = {Silverman, J. H.},
  title      = {The Arithmetic of Dynamical Systems},
  series     = {Graduate Texts in Mathematics},
  volume     = {241},
  publisher  = {Springer, New York},
  year       = {2007},
  pages      = {x+511},
  isbn       = {978-0-387-69903-5},
  mrclass    = {11-02 (11-01 11G05 11G07 11G50 37-02 37F10)},
  mrnumber   = {2316407},
  mrreviewer = {Thomas\ Ward},
  doi        = {10.1007/978-0-387-69904-2},
  url        = {https://doi.org/10.1007/978-0-387-69904-2}
}

@article{MR0175882,
  author     = {Schinzel, A. and Zassenhaus, H.},
  title      = {A refinement of two theorems of {K}ronecker},
  journal    = {Michigan Math. J.},
  fjournal   = {Michigan Mathematical Journal},
  volume     = {12},
  year       = {1965},
  pages      = {81--85},
  mrclass    = {10.65},
  mrnumber   = {175882},
  mrreviewer = {R.\ Ayoub},
  url        = {http://projecteuclid.org/euclid.mmj/1028999247}
}

@book {MR2919975,
    AUTHOR = {Cox, David A.},
     TITLE = {Galois theory},
    SERIES = {Pure and Applied Mathematics (Hoboken)},
   EDITION = {Second},
 PUBLISHER = {John Wiley \& Sons, Inc., Hoboken, NJ},
      YEAR = {2012},
     PAGES = {xxviii+570},
      ISBN = {978-1-118-07205-9},
   MRCLASS = {12F10 (00-01 01A05)},
  MRNUMBER = {2919975},
       DOI = {10.1002/9781118218457},
       URL = {https://doi.org/10.1002/9781118218457},
}

@misc{habegger2026dynamicalcanonicalheightsfinite,
      title={Dynamical Canonical Heights and Finite Trees}, 
      author={Habegger, Philipp and Schmidt, Harry},
      year={2026},
      eprint={2608.00774},
      archivePrefix={arXiv},
      primaryClass={math.NT},
      url={https://arxiv.org/abs/2608.00774}, 
}

@misc{cangini2026lowerboundcallsilvermanheight,
      title={A lower bound for the Call-Silverman height in cyclotomic extensions}, 
      author={Cangini, Alessio},
      year={2026},
      eprint={2609.01747},
      archivePrefix={arXiv},
      primaryClass={math.NT},
      url={https://arxiv.org/abs/2609.01747}, 
}

@article {MR2029512,
    AUTHOR = {Silverman, Joseph H.},
     TITLE = {A lower bound for the canonical height on elliptic curves over
              abelian extensions},
   JOURNAL = {J. Number Theory},
  FJOURNAL = {Journal of Number Theory},
    VOLUME = {104},
      YEAR = {2004},
    NUMBER = {2},
     PAGES = {353--372},
      ISSN = {0022-314X,1096-1658},
   MRCLASS = {11G50 (11G05 11G10 14G25 14K15)},
  MRNUMBER = {2029512},
MRREVIEWER = {Robin\ de Jong},
       DOI = {10.1016/j.jnt.2003.07.001},
       URL = {https://doi.org/10.1016/j.jnt.2003.07.001},
}

@article {MR2067482,
    AUTHOR = {Baker, Matthew H. and Silverman, Joseph H.},
     TITLE = {A lower bound for the canonical height on abelian varieties over abelian extensions},
   JOURNAL = {Math. Res. Lett.},
  FJOURNAL = {Mathematical Research Letters},
    VOLUME = {11},
      YEAR = {2004},
    NUMBER = {2-3},
     PAGES = {377--396},
      ISSN = {1073-2780},
   MRCLASS = {11G50 (11G10 14K15)},
  MRNUMBER = {2067482},
MRREVIEWER = {Nicolas\ Ratazzi},
       DOI = {10.4310/MRL.2004.v11.n3.a10},
       URL = {https://doi.org/10.4310/MRL.2004.v11.n3.a10},
}

@article {MR2098701,
    AUTHOR = {Ratazzi, Nicolas},
     TITLE = {Th\'eor\`eme de {D}obrowolski-{L}aurent pour les extensions
              ab\'eliennes sur une courbe elliptique \`a{} multiplication
              complexe},
   JOURNAL = {Int. Math. Res. Not.},
  FJOURNAL = {International Mathematics Research Notices},
      YEAR = {2004},
    NUMBER = {58},
     PAGES = {3121--3152},
      ISSN = {1073-7928,1687-0247},
   MRCLASS = {11G15 (11G05)},
  MRNUMBER = {2098701},
MRREVIEWER = {Matthew\ H.\ Baker},
       DOI = {10.1155/S1073792804140518},
       URL = {https://doi.org/10.1155/S1073792804140518},
}
\end{document}